\documentclass[11pt]{article}
\usepackage{amsmath}
\usepackage{amsfonts}
\usepackage{color}
\usepackage{latexsym}
\usepackage{tablefootnote}
\usepackage{epsfig}
\usepackage{multirow}
\usepackage{float}
\usepackage{array}
\usepackage{amssymb,amsthm}
\usepackage{hyperref}    
\usepackage{algorithm2e} 
\usepackage{cleveref}    

\newcommand*{\R}{\mathbb{R}}

\newcommand{\ie}{\emph{i.e.,}}

\newcommand{\ascal}[2]{\left\langle #1, #2 \right\rangle}
\renewcommand{\aa}[1]{\left\{ #1 \right\}}

\newcommand{\red}[1]{{\color{red}{#1}}}

\definecolor{antiquefuchsia}{rgb}{0.57, 0.36, 0.51}
\definecolor{MyViolet}{rgb}{0.45,0.08,0.95}
\definecolor{MyBrown}{rgb}{0.45,0.08,0}
\definecolor{MyDarkBlue}{rgb}{0,0.08,0.45}

\newcommand{\oldvl}[1]{\red{#1}}

\usepackage{makecell,booktabs}
\usepackage[version=4]{mhchem}   
\usepackage[strict]{changepage}

\usepackage{amsmath}
\usepackage{amsfonts}
\usepackage{latexsym}
\usepackage{amssymb}

\newtheorem{thm}{Theorem}[section]
\newtheorem{theorem}[thm]{Theorem}

\newtheorem{lemma}[thm]{Lemma}
\newtheorem{prop}[thm]{Proposition}

\newtheorem{rem}[thm]{Remark}
\newtheorem{remark}[thm]{Remark}

\newcommand{\beq}{\begin{equation}}
\newcommand{\eeq}{\end{equation}}
\newcommand{\beqa}{\begin{eqnarray}}
\newcommand{\eeqa}{\end{eqnarray}}
\newcommand{\beqas}{\begin{eqnarray*}}
\newcommand{\eeqas}{\end{eqnarray*}}
\newcommand{\bi}{\begin{itemize}}
\newcommand{\ei}{\end{itemize}}

\newcommand{\nn}{\nonumber}

\newcommand{\dom}{\mathrm{dom}\,}

\newcommand{\Argmin}{\mathrm{Argmin}\,}
\newcommand{\argmin}{\mathrm{argmin}\,}

\newcommand{\bConv}[1]{\overline{\mbox{\rm Conv}}\,(\R^{#1})}

\newcommand{\mConv}[1]{\overline{\mbox{\rm Conv}}_\mu\,(\R^{#1})}

\begin{document}
	\title{Bundle methods with quadratic cuts for deterministic and stochastic strongly convex optimization problems}
\date{}
 
	 \maketitle

\begin{center}
\begin{tabular}{ccc}
\begin{tabular}{c}
Vincent Guigues\\
School of Applied Mathematics, FGV\\
Praia de Botafogo, Rio de Janeiro, Brazil\\
{\tt vincent.guigues@fgv.br}
\end{tabular}&
&
\begin{tabular}{c}
Adriana Washington\\
School of Applied Mathematics, FGV\\
Praia de Botafogo, Rio de Janeiro, Brazil\\
{\tt adriana.washington07@gmail.com}
\end{tabular}
\end{tabular}
\end{center}

	\begin{abstract}
We introduce two new methods for deterministic convex optimization problems: QCC (Quadratic Cuts for Convex optimization) and QB (Quadratic Bundle method). We prove the complexity of these
methods for composite optimization problems
which are the sum of a convex function $\tilde h$ and
of a strongly convex function $\tilde f$ with parameter $\mu$. 
These methods use as building blocks
quadratic approximations of the strongly convex function $\tilde f$ where the quadratic terms
are of form $\frac{\mu}{2}\|\cdot-x_i\|^2$
for trial points $x_i$ computed along iterations (when $\mu=0$ the building blocks are linear approximations).
We extend the idea of using quadratic approximations to pieces of the objective
for some
multistage stochastic optimization problems
which have strongly convex recourse functions
that we approximate as a maximum of quadratic
cuts. We call DASC (Dynamic Approximation for Strongly Convex optimzation) the corresponding
optimization method. When the cuts are linear, the method
boils down to the popular Stochastic Dual Dynamic Programming (SDDP) method.
We provide conditions ensuring strong convexity
of the recourse functions and prove the convergence of DASC. 
Numerical experiments
illustrate the performance and correctness of
DASC, with DASC being much quicker than SDDP for large values of the constants
of strong convexity.\\
\end{abstract}

		
	 
	\section{Introduction}\label{sec:intro}

 In this paper we introduce new optimization methods for strongly convex optimization, both for deterministic problems and
 for multistage stochastic programs. 
 In the deterministic setting, 
 given $\tilde f, \tilde h: \R^{n} \rightarrow \R\cup \{ +\infty \}$ proper lower semi-continuous  convex functions such that
	$ \dom \tilde h \subseteq \dom \tilde f$ and
	$\tilde f$ $\mu$-convex, the problem we want to solve is
 \begin{equation}\label{eq:ProbIntro1}
	\phi_{*}:=\min \left\{\phi(x):=\tilde f(x)+\tilde h(x): x \in \R^n\right\}.
	\end{equation}
 
 The building blocks of many optimization
methods for convex optimization are linearizations of the objective function and/or of the functions in the constraints which can be combined in many different ways to obtain models for the
functions that these linearizations approximate.
The simplest and oldest of these methods is Kelley's algorithm \cite{kelley66}. This reference was followed by a huge literature
also exploiting the aforementioned
linearizations.
Important advances 
include
bundle methods \cite{lemarechal1995new, kiwiel1986method, kiwiel2006methods,kiwiel1995proximal, kiwiel2000efficiency}
and level bundle methods
\cite{lemarechal1995new}.
A bundle method allowing one cut only has also
been proposed in \cite{bundlegpbstochastic} for stochastic 
composite optimization problems.

In the deterministic setting, our methods
also belong to the class of universal methods that do
not need the knowledge of any problem parameters
(such as Lipschitz constants, diameters of sets)
and to the class of methods for composite optimization that
minimize a sum of functions.
The first universal methods for composite optimization
of form \eqref{eq:ProbIntro1}
under the condition that one of the functions is Hölder continuous 
have been presented in
\cite{nesterov2015universal}
and 
\cite{lan2015bundle}.
Additional universal methods
for solving \eqref{eq:ProbIntro1} have been studied in \cite{mishchenko2020adaptive,zhou2024adabb, alamo2019gradient,alamo2022restart,alamo2019restart,aujol2023parameter,aujol2023fista,aujol2022fista,fercoq2019adaptive,lan2023optimal,nesterov2013gradient,renegar2022simple, universalguilimont}.

In the stochastic setting, we consider
Multistage Stochastic convex Programs (MSPs).
For such problems, the most popular solution method is the Stochastic Dual Dynamic Programming (SDDP) method, introduced in \cite{pereira}. SDDP is a sampling-based extension of the 
Nested Decomposition algorithm \cite{birgemulti} which builds policies for some multistage stochastic convex problems.
It has been used to solve many real-life problems and several extensions of the method have been considered 
such as DOASA \cite{philpot}, CUPPS \cite{chenpowell99}, ReSA \cite{resa}, AND \cite{birgedono}, risk-averse variants (\cite{guiguesrom10}, \cite{guiguesrom12}, \cite{philpmatos}, 
\cite{guiguescoap2013}, \cite{shapsddp}, \cite{shadenrbook}, \cite{kozmikmorton}), inexact variants (\cite{guisvmont}),
variants for problems with a random
number of stages \cite{gui21r}, or statistical
upper bounds \cite{guishapcheng}, see also references therein.
SDDP builds approximations for the cost-to-go functions which take the form of a maximum of
affine functions called cuts.\\

\par {\textbf{Our contributions, both for problems of form \eqref{eq:ProbIntro1} and for multistage stochastic optimization problems, are given below.}}\\ 

\par (A) {\textbf{For deterministic problems of form \eqref{eq:ProbIntro1}.}} To solve \eqref{eq:ProbIntro1}, we introduce two new methods for convex optimization problems: QCC (Quadratic Cuts for Convex optimization) and QB (Quadratic Bundle method). We prove the complexity of these
methods for composite optimization problems
of form \eqref{eq:ProbIntro1}
which are the sum of a convex function $\tilde h$ and
of a $\mu$-convex function $\tilde f$. A new feature of our optimization methods is that their building blocks are quadratic
approximations of the objective obtained at points computed
along the iterations of the methods, instead of linearizations. To the best of our knowledge, this idea has not been used in previous publications. Moreover, we use a model update function which allows us to combine the quadratic
approximations in many different ways to obtain models for the approximated functions.
Three special cases of these models are (A) a unique quadratic function 
which is a convex combination of all quadratic approximations, (B) a maximum of two
quadratic cuts, and (C) the maximum of all quadratic approximations
computed so far.\\

\par {\textbf{(B) For multistage stochastic programs.}}
We propose an extension of SDDP method called DASC (Dynamic Approximation for Strongly Convex problems), which is a  
 Decomposition Algorithm for multistage stochastic programs having Strongly Convex cost functions.
Similarly to SDDP, at each iteration the algorithm computes in a forward pass a sequence of trial points which are used
in a backward pass to build lower bounding functions called cuts. However, contrary to SDDP where cuts are affine functions,
the cuts computed with DASC   are quadratic functions
and therefore
the cost-to-go functions are approximated by a maximum of quadratic functions. We also show the convergence of DASC.\\

\par {\textbf{(C) For multistage stochastic programs.}} We present numerical simulations showing the correctness of DASC and comparing the performance on DASC and SDDP on a simple MSP.\\

\par The outline of the paper is as follows. 
Section \ref{sec:framework} presents
our Quadratic Cuts for Convex optimization method and shows its complexity.
Section \ref{sec:qb} describes our Quadratic Bundle method and gives its complexity.
In Section \ref{sec:scvaluefunc}, we give in Proposition \ref{strongconvvfunc} a 
simple condition ensuring that the value function of a convex optimization problem is strongly convex.
In Section \ref{sec:pbformass}, we introduce the class of
multistage stochastic optimization  problems to which DASC applies and 
the necessary assumptions. DASC algorithm, which is based on Proposition \ref{strongconvvfunc}, is given in Section \ref{sec:dasc}, while 
convergence  
of DASC is shown in Section \ref{convanalysis}.
Numerical simulations comparing DASC and SDDP are given in Section \ref{sec:num} while concluding remarks and possible extensions are discussed in Section \ref{sec:conc}.

    \subsection{Basic definitions and notation.} Let $\R$ denote the set of real numbers.
    Let $ \R_+ $ and $ \R_{++} $ denote the set of non-negative real numbers and the set of positive real numbers, respectively.
	Let $\R^n$ denote the standard $n$-dimensional Euclidean space equipped with  inner product and norm denoted by $\left\langle \cdot,\cdot\right\rangle $
	and $\|\cdot\|$, respectively. 
	Let $\log(\cdot)$ denote the natural logarithm.

	Let $\Psi: \R^n\rightarrow (-\infty,+\infty]$ be given. Let $\dom \Psi:=\{x \in \R^n: \Psi (x) <\infty\}$ denote the effective domain of $\Psi$.
 We say that
	$\Psi$ is proper if $\dom \Psi \ne \emptyset$.
	A proper function $\Psi: \R^n\rightarrow (-\infty,+\infty]$ is $\mu$-convex for some $\mu \ge 0$ if
	$$
	\Psi(\alpha z+(1-\alpha) z')\leq \alpha \Psi(z)+(1-\alpha)\Psi(z') - \frac{\alpha(1-\alpha) \mu}{2}\|z-z'\|^2
	$$
	for every $z, z' \in \dom \Psi$ and $\alpha \in [0,1]$.
 The set of all proper lower semicontinuous $\mu$-convex functions is denoted by $\mConv{n}$.
 When $\mu=0$, we simply denote
    $\mConv{n}$ by $\bConv{n}$.
	For $\varepsilon \ge 0$, the \emph{$\varepsilon$-subdifferential} of $ \Psi $ at $z \in \dom \Psi$ is denoted by
	\begin{equation}
        \partial_\varepsilon \Psi (z):=\left\{ s \in\R^n: \Psi(z')\geq \Psi(z)+\left\langle s,z'-z\right\rangle -\varepsilon, \forall z'\in\R^n\right\}.
        \end{equation}
	The subdifferential of $\Psi$ at $z \in \dom \Psi$, denoted by $\partial \Psi (z)$, is by definition the set  $\partial_0 \Psi(z)$.

\if{

\subsection{A motivating algorithm}

We consider the following optimization problem 
\begin{equation}
\label{pb:initial}
    \min \left\{f(x) \;|\; x \in X\right\},
\end{equation}
where $f$ is a convex lower-semicontinuous function, subdifferentiable on $X$ a compact convex subset of $\R^n$.
The Kelley cutting plane method applied to Problem~\eqref{pb:initial} consists in approximating the convex function $f$ 
by a polyhedral approximation that we iteratively refine. 
More precisely, let $f'$ define a selection of $\partial f$, \ie a function such that, for all $x\in X$, $f'(x) \in \partial f(x)$.
Further,  denote $\ell_f(\cdot,x):= f(x) + \ascal{f'(x)}{\cdot - x}$ , for any $x \in X$, the tangent of $f$ at $x$ of slope $f'(x)$.
In particular, as $f$ is convex we have $\ell_f(\cdot,x) \leq f$, and any trial point $x_k \in X$ define a so-called (affine) cut $\ell_f(\cdot,x_k) \leq f$.
Then, the Kelley algorithm, given a collection of trial points $(x^\kappa)_{\kappa \in [k]}$ at iteration $k$,
defines  a model $\Gamma^k$ of $f$ as the maximum of past cuts, \ie $ \max_{\kappa \in [k]} \ell_f(\cdot,x^k) \leq f$. 
The next trial point $x^{k+1}$ is defined as a minimum of $\Gamma^k$ over $X$. 
This is detailed in~\cref{alg:Kelley}, where we can note that, at each iteration we have a candidate solution $x_k  \in X$, and an optimality gap $t_k \geq 0$.

\begin{algorithm}
\KwData{$x_1 \in X$}
\KwResult{approximate solution $x_n$, gap $t_n$}
$\Gamma^1 = \ell_f(\cdot,x_1) $;\\
\For{$k \in [n-1]$}
{
    $x_{k+1} \in \argmin_{x \in X} \Gamma_k(x) $ ;\\
    $t_{k+1} = f(x_{k+1}) - \Gamma_k(x_{k+1})$ ;\\
   $\Gamma_{k+1} \leftarrow \max \aa{\Gamma_{k}, \ell_f(\cdot,x_{k+1})}$;
}
\caption{Kelley's cutting plane algorithm}\label{alg:Kelley}
\end{algorithm}

This algorithm is well-known and studied but only leverages the convexity of $f$. 
Assume now that we want to minimize $\mu$-convex function $h$ over $X$, we can then use quadratic cuts 
\begin{equation}
\label{eq:qf}
    q_h= \underbrace{h(x) + \ascal{h'(x)}{\cdot - x}}_{\ell_h(\cdot,x)} + \frac{\mu}{2}\| \cdot - x\|^2 \leq h
\end{equation}
instead of the linear cuts 
$\ell_h(\cdot,x)$, yielding a new, to our knowledge, strongly-convex Kelley algorithm 
\begin{algorithm}
\KwData{$x_1 \in X$}
\KwResult{approximate solution $x_n$, gap $t_n$}
$\Gamma^1 = q_h(\cdot,x_1) $;\\
\For{$k \in [n-1]$}
{
    $x_{k+1} \in \argmin_{x \in X} \Gamma_k(x) $ ;\\
    $t_{k+1} = f(x_{k+1}) - \Gamma_k(x_{k+1})$ ;\\
   $\Gamma_{k+1} \leftarrow \max \aa{\Gamma_{k}, q_h(\cdot,x_{k+1})}$;
}
\caption{Cutting quadratic algorithm}\label{alg:Kelley_quad}
\end{algorithm}

To study the complexity of this algorithm we suggest seeing it as a composite strongly convex optimization problem. 
First, we define $f = \delta_X$ the indicator function of $X$, such that $f(x)=0$ if $x\in X$, and $f(x)=+\infty$ otherwise.
As $X$ is convex, so is $f$.
Now consider 

}\fi

\section{QCC method and its complexity}\label{sec:framework}

Let $\tilde f, \tilde h: \R^{n} \rightarrow \R\cup \{ +\infty \}$ be proper lower semi-continuous  convex functions such that
	$ \dom \tilde h \subseteq \dom \tilde f$ and
	$\tilde f$ is $\mu$-convex
	for some $ \mu\geq 0 $. 
 In this section, we are interested in solving problem \eqref{eq:ProbIntro1}.

We assume that we have for $\tilde f$
a subgradient oracle, i.e.,
		a function $\tilde f':\dom \tilde h \to \R^n$
		satisfying $\tilde f'(x) \in \partial \tilde f(x)$ for every $x \in \dom \tilde h$.
We also assume that the set of optimal solutions $X^*$ of
		problem \eqref{eq:ProbIntro1} is nonempty.
 
\subsection{QCC method}\label{subsec:update}
QCC method computes at iteration
$j$ an approximate solution of problem 
\eqref{eq:ProbIntro1} given by
$$
x_{j} =\underset{u\in \R^n}\argmin
	     \;\Gamma_{j}(u),
$$
where $\Gamma_{j}$ is a model
for $\tilde f+\tilde h$ satisfying
$$
\Gamma_{j}\in \mConv{n} \mbox{ and }\Gamma_{j}\le  \tilde f+\tilde h.
$$

As we shall see, model $\Gamma_j$ is actually
the sum of $\tilde h$ and of a model for $\tilde f$.
Therefore, the composite QCC method 
applies to problems where $\tilde f$ in the objective
may be a {\em{complicate}} function to minimize
while $\tilde h$ is a simple function to minimize.
By convexity of $\tilde f$, the linear
approximation 
$$
\ell_{\tilde f}(\cdot,x)=
\tilde f(x)
+\langle \tilde f'(x),\cdot-x\rangle
$$
of $\tilde f$ at $x$
satisfies $\ell_{\tilde f}(\cdot,x) \leq \tilde f$
and such linearizations are the basis of many
methods for convex optimization where $\tilde f$ is in the objective function (or even in the constraints).
Since $\tilde f$ is strongly convex with parameter
$\mu\geq0$, the quadratic
approximation 
\begin{equation}\label{quadapprox}
q_{\tilde f}(u;x) =\tilde f(x)
+\langle \tilde f'(x),u-x\rangle + \frac{\mu}{2}\|u-x\|^2
=\ell_{\tilde f}(u,x) + \frac{\mu}{2}\|u-x\|^2
\end{equation}
of $\tilde f$ at $x$ provides a better approximation
of $\tilde f$ than $\ell_{\tilde f}(u;x)$ and is still a lower bounding function
for $\tilde f$:
\begin{equation}\label{quadminorant}
\tilde f(u) \geq q_{\tilde f}(u;x) \geq \ell_{\tilde f}(u,x)
\end{equation}
for every $u \in \mathbb{R}^n$.
These quadratic functions will be combined to provide
models for $\tilde f$ in QCC.

  For an iteration $j \geq 2$, the model
  $\Gamma_j$ is obtained using a
  Quadratic Update function 
  $$
  \mbox{QU}(\Gamma_{j-1},x_{j-1},\tau)$$ depending
  on three inputs: the previous model
  $\Gamma_{j-1}$, the previous approximate
  solution $x_{j-1}$, and a parameter
  $\tau \in (0,1)$. 
We now describe this function and make a few comments
on this model update function.\\
\noindent\rule[0.5ex]{1\columnwidth}{1pt}
	
	QU$(\Gamma,x,\tau)$

    \noindent\rule[0.5ex]{1\columnwidth}{1pt}
    {\bf Input:} Let $\tau \in (0,1)$ and
    $(x,\Gamma) \in \R^n \times \mConv{n}$ such that $\Gamma\le \phi$ and 
    $x=\underset{u\in \R^n}\argmin \;\Gamma (u)$.

\noindent Find function $\Gamma^+$ such that
\begin{equation}\label{def:Gamma}
	    \Gamma^+ \in \mConv{n}, \qquad \tau \bar \Gamma(\cdot)  + (1-\tau) [q_{\tilde f}(\cdot;x)+\tilde h(\cdot)] \le \Gamma^+(\cdot) \le \phi(\cdot),
	\end{equation}
   where $q_{\tilde f}(\cdot;x)$ is as in \eqref{quadapprox} and
   $\bar \Gamma(\cdot) $ is such that
\begin{equation}\label{def:bar Gamma}
	\bar \Gamma \in \mConv{n}, \quad
 \bar \Gamma(x) = \Gamma(x), \quad 
	x = \underset{u\in \R^n}\argmin \;\bar \Gamma (u).
	\end{equation}
     {\bf Output:} $\Gamma^+$.
     
    \noindent\rule[0.5ex]{1\columnwidth}{1pt}

	Clearly, the above update scheme does not completely determine $\Gamma^+$
	but rather gives minimal conditions on it which are suitable for the complexity analysis
	of QCC developed in the next section.
A natural choice for the first model $\Gamma_1$ is $\Gamma_1 = \tilde h + q_{\tilde f}(,x_0)$ for a given first (arbitrary) trial point $x_0$.
The new model $\Gamma^{+}$ must be above a convex combination of
two functions, with $\tau \in (0,1)$ being the parameter defining this convex combination.
The first function in this convex combination is
$\tilde h$ plus
the last computed quadratic cut $q_{\tilde f}(\cdot,x)$.
The second function in this convex combination is
another model
$\bar \Gamma$ satisfying some conditions (this model can be as a special case 
the previous model $\Gamma$).
The models $\Gamma$
are convex and lower
bounding functions for $\phi$.
 
	We now describe three concrete
 update schemes (E1), (E2), and (E3) which are special ways
 of implementing QU function. 
	\begin{itemize} 
     \item [(E1)] {\bf One-cut scheme}: 
    	    This scheme obtains $\Gamma^+$ as
	    \begin{equation}\label{eq:affine}
        \Gamma^+(\cdot)= \Gamma^+_\tau(\cdot) := \tau \Gamma(\cdot) + (1-\tau) [q_{\tilde f}(\cdot;x)+\tilde h(\cdot)]
        \end{equation}
        where $x = \underset{u\in \R^n}\;\argmin \Gamma (u)$ and $\tau\in (0,1)$. 
        Clearly, $\Gamma^+$ satisfies 
        \eqref{def:Gamma} 
        with $\bar \Gamma = \Gamma$ satisfying
        \eqref{def:bar Gamma}. 
    Also, if
        $\Gamma$ is the sum of $\tilde h$ and
        a quadratic function underneath $\tilde f$,
        then so is $\Gamma^+$.	   
Such is the case if the first model 
$\Gamma_1$ is chosen to be 
$q_{\tilde f}(u;x_0)+\tilde h(u)$
     in which case
     $\Gamma_j$ is of the form
	\begin{equation}\label{eq:Gamma-form}
	    \Gamma_j(\cdot) = \sum_{i=0}^{j-1} \alpha_i q_{\tilde f}(u;x_i)+\tilde h(\cdot)
	\end{equation}
	where 
  $\{\alpha_0,\ldots,\alpha_{j-1}\} \subset \R_{++}$ are scalars such that
	$\displaystyle \sum_{i=0}^{j-1} \alpha_i =1$.
 This amounts to use a very simple approximaton
 of $\tilde f$ which is a quadratic function, itself
 expressed as a convex combination of quadratic approximations
 of $\tilde f$ computed along iterations.
        \item [(E2)] \textbf{Two-cuts scheme:} 
        For this scheme, it is assumed that
        $\Gamma$
        has the form 
        \begin{equation}\label{def:Gamma-E2}
            \Gamma=\max \{Q_{\tilde f},q_{\tilde f}(\cdot;x^-)\}+{\tilde h}
        \end{equation}
        where $\tilde h \in \bConv{n}$ and 
        $Q_{\tilde f}$ is a quadratic function satisfying $Q_{\tilde f}\le \tilde f$. In view of \
        $$
        x= \underset{u\in \R^n}\argmin \;\Gamma (u)=\underset{u\in \R^n}\argmin \max \{Q_{\tilde f}(u),q_{\tilde f}(u;x^-)\}+{\tilde h}(u),
        $$
        there exists
        $\theta \in [0,1]$ such that
        \begin{align}
            &0 \in    \partial \tilde h(x) 
+ \theta \nabla Q_{\tilde f}(x) + (1-\theta) \nabla q_{\tilde f}(x;x^-), \label{theta1} \\
            &\theta Q_{\tilde f}(x) + (1-\theta) q_{\tilde f}(x;x^-) = \max \{Q_{\tilde f}(x),q_{\tilde f}(x;x^-)\}.\label{theta2}
        \end{align}        
        The scheme then sets
        \begin{equation}\label{def:Af+}
         Q_{\tilde f}^+(\cdot) :=  \theta Q_{\tilde f}(\cdot) + (1-\theta) q_{\tilde f}(\cdot;x^-)
     \end{equation}
     and outputs the function $\Gamma^+$
      defined as
\begin{equation}\label{eq:G-agg}
		    \Gamma^+(\cdot)  := 
		    \max\{Q_{\tilde f}^+(\cdot),q_{\tilde f}(\cdot;x)\} + {\tilde h}(\cdot).
	    \end{equation}
    We now show that this model
     $\Gamma^+$ satisfies \eqref{def:Gamma} 
        with $\bar \Gamma = Q_{\tilde f}^+ + \tilde h$ satisfying
        \eqref{def:bar Gamma}. 
        It is clear that $\Gamma^{+},\bar \Gamma \in \mConv{n}$. Moreover, $Q_{\tilde f}(\cdot) \leq \tilde f(\cdot)$
        and $q_{\tilde f}(\cdot;x^{-}) \leq \tilde f(\cdot)$ implies that $\Gamma^{+} \leq \phi$. Next, relation
        $\tau \bar \Gamma(\cdot)  + (1-\tau) [q_{\tilde f}(\cdot;x)+{\tilde h}(\cdot)] \le \Gamma^+(\cdot)$
        is an immediate consequence
        of the definitions of
        $\bar \Gamma$ and
        $\Gamma^+$ and we have shown that
        $\Gamma^+$ satisfies \eqref{def:Gamma}.
     The relation $\Gamma(x)=\bar \Gamma(x)$
     immediately follows from
     the definition $\bar \Gamma = 
     Q_{\tilde f}^{+} + {\tilde h}$
     of $\bar \Gamma$ and from relations
     \eqref{def:Gamma-E2}, \eqref{theta2},
     and \eqref{def:Af+}. Finally, by
     \eqref{theta1} we 
     deduce
$0 \in \partial \bar \Gamma(x)$, which can be written $x= \underset{u\in \R^n}\argmin \;\bar \Gamma (u)$ and we have
shown that $\bar \Gamma$ satisfies \eqref{def:bar Gamma}.  Therefore, Example (E2) shows
how to obtain a model of
$\tilde f$ made of two quadratic cuts only.

	    \item [(E3)] {\textbf{Multiple-cuts scheme:}} For this scheme, it is assumed that $\Gamma$ is of the form
	    $\Gamma=\Gamma(\cdot;B)$ where
	    $B \subset \R^n$ is a finite set and $\Gamma(\cdot;B)$ is defined by
$$
\Gamma(\cdot;B)=
\max(q_{\tilde f}(\cdot;b):b \in B)+{\tilde h}(\cdot).
$$
This scheme 
      chooses the next set $B^+$ so that
	    \begin{equation}\label{eq:C+}
        B(x) \cup \{x\} \subset B^+ \subset B \cup \{x\}
        \end{equation}
        where
        \begin{equation}\label{def:C+}
            B(x) := \{ b \in B : q_{\tilde f}(x;b)+{\tilde h}(x) = \Gamma(x) \},
        \end{equation}
         and then
	    outputs $\Gamma^+ = \Gamma(\cdot;B^+)$. We now show that this model
     $\Gamma^+$ satisfies \eqref{def:Gamma} with $\bar \Gamma =\Gamma(\cdot;B(x))$ satisfying
        \eqref{def:bar Gamma}. 
As before, it is clear that $\Gamma^{+},\bar \Gamma \in \mConv{n}$
and that $\Gamma^{+} \leq \phi$
since all quadratic functions $q_{\tilde f}(\cdot;z)$ satisfy 
        $q_{\tilde f}(\cdot;z) \leq {\tilde f}(\cdot)$ for all 
        $z$. From the definitions of
        $\bar \Gamma$, $\Gamma^{+}$,
        and $B^+$, we have
        $$
        \Gamma^{+}(\cdot) \geq 
        \max(\bar \Gamma(\cdot), q_{\tilde f}(\cdot;x)+{\tilde h}),
        $$
which implies 
$\tau \bar \Gamma(\cdot)  + (1-\tau) [q_{\tilde f}(\cdot;x)+{\tilde h}(\cdot)] \le \Gamma^+$
and we have shown
that $\Gamma^+$ satisfies
\eqref{def:Gamma}. Next, using the
fact that $\partial \Gamma(x)=\partial  \bar \Gamma(x)$ and the relation
$0 \in \partial  \Gamma(x)$, we deduce
that $x$ minimizes $\bar \Gamma$. Finally,
$$
\bar \Gamma(x)=\max(q_{\tilde f}(x;b):b  \in B(x))+\tilde h(x)=\max(q_{\tilde f}(x;b):b  \in B)+{\tilde h}(x)=\Gamma(x)
$$
and therefore we have shown that $\bar \Gamma$
satisfies \eqref{def:bar Gamma}. Taking $B^{+}=B\cup\{x\}$ for every iteration, we obtain a model for
$\tilde f$ given by the maximum of all quadratic cuts previously computed.

	\end{itemize}

We are now in a position to describe the QCC framework, which is
given below. It uses
model update function QU.
	
	\noindent\rule[0.5ex]{1\columnwidth}{1pt}
	
	QCC
	
	\noindent\rule[0.5ex]{1\columnwidth}{1pt}
	\begin{itemize}
		\item [0.] Let $ x_0\in \dom h $, $ \bar \varepsilon>0 $ and $\tau \in (0,1)$ be given.
  \if{
		such that
        \begin{equation}\label{rel:tau1} 
	        \frac{\tau}{1-\tau} \ge \frac{8 T_{\bar \varepsilon}^2}{\mu\bar \varepsilon}
	    \end{equation}
		where $T_{\bar \varepsilon}$ is as in \eqref{def:T},
  }\fi
  Set  
		$y_0=x_0$ and $j=1$;
		\item[1.] if $j=1$ then find $\Gamma_{j}$ such that
		\begin{equation}\label{ineq:require}
		    \Gamma_{j}\in \mConv{n}, \quad q_{\tilde f}(\cdot;x_{j-1})+
      {\tilde h} \le \Gamma_{j}\le  \phi;
		\end{equation}
		else, let 
  $\Gamma_{j}=$QU$(\Gamma_{j-1},x_{j-1},\tau)$;
		
	    \item[2.]
	    compute 
		\begin{equation}
	    x_{j} =\underset{u\in \R^n}\argmin \;
	     \Gamma_{j}(u),  \label{def:xj}
	    \end{equation}
		 choose 
   \begin{equation}\label{eq:yj-intuition}
    y_j \in \Argmin \left\lbrace \phi(x) :
		x \in \{x_0,x_1,\ldots,x_j\}
		\right\rbrace, 
\end{equation}
		and set  
		\begin{equation}\label{ineq:hpe1}
		    m_{j}=\Gamma_{j}(x_{j}), \qquad t_{j} = \phi(y_{j}) - m_{j};
		\end{equation}
		\item[3.] set $j\leftarrow j+1$ and go to Step 1.
	\end{itemize}
	\rule[0.5ex]{1\columnwidth}{1pt}

\if{
For iteration $j=1$, it is  required that
		$$
		    \Gamma_{j}\in \mConv{n} \mbox{ and } \ell_f(\cdot;x_{j-1})+h \le \Gamma_{j}\le  \phi.
		$$

  For an iteration $j \geq 2$, the model
  $\Gamma_j$ is obtained using a
  Model Update function MU$(\Gamma_{j-1},x_{j-1},\tau)$ depending
  on three inputs: the previous model
  $\Gamma_{j-1}$, the previous approximate
  solution $x_{j-1}$, and a parameter
  $\tau \in (0,1)$. 
We now describe this function and make a few comments
on this model update function.\\
\noindent\rule[0.5ex]{1\columnwidth}{1pt}
	
	MU$(\Gamma,x,\tau)$

    \noindent\rule[0.5ex]{1\columnwidth}{1pt}
    {\bf Input:} Let $\tau \in (0,1)$ and
    $(x,\Gamma) \in \R^n \times \mConv{n}$ such that $\Gamma\le \phi$.

\noindent Find function $\Gamma^+$ such that
\begin{equation}\label{def:Gamma}
	    \Gamma^+ \in \mConv{n}, \qquad \tau \bar \Gamma(\cdot)  + (1-\tau) [\ell_f(\cdot;x)+h(\cdot)] \le \Gamma^+(\cdot) \le \phi(\cdot),
	\end{equation}
   where $\ell_f(\cdot;\cdot)$ is as in \eqref{def:gamma} and
   $\bar \Gamma(\cdot) $ is such that
\begin{equation}\label{def:bar Gamma}
	\bar \Gamma \in \mConv{n}, \quad
 \bar \Gamma(x) = \Gamma(x), \quad 
	x = \underset{u\in \R^n}\argmin \;\bar \Gamma (u).
	\end{equation}
     {\bf Output:} $\Gamma^+$.
     
    \noindent\rule[0.5ex]{1\columnwidth}{1pt}

	Clearly, the above update scheme does not completely determine $\Gamma^+$
	but rather gives minimal conditions on it which are suitable for the complexity analysis
	of QCC.
	
	We now describe three concrete
 update schemes (E1), (E2), and (E3) which are special ways
 of implementing MU function. 
	\begin{itemize} 
     \item [(E1)] {\bf One-cut scheme}: 
    	    This scheme obtains $\Gamma^+$ as
	    \begin{equation}\label{eq:affine}
        \Gamma^+= \Gamma^+_\tau := \tau \Gamma + (1-\tau) [\ell_f(\cdot;x)+h]
        \end{equation}
        where $x = \underset{u\in \R^n}\argmin \Gamma (u)$ and $\tau\in (0,1)$ depends on $(L_f,M_f,\mu)$. 
        Clearly, $\Gamma^+$ satisfies 
        \eqref{def:Gamma} 
        with $\bar \Gamma = \Gamma$ satisfying
        \eqref{def:bar Gamma}. 
    Also, if
        $\Gamma$ is the sum of $h$ and
        an affine function underneath $f$,
        then so is $\Gamma^+$.	   
Such is the case if the first model 
$\Gamma_1$ is chosen to be 
$\ell_f(\cdot;x_0)+h$
     in which case
     $\Gamma_j$ is of the form
	\begin{equation}\label{eq:Gamma-form}
	    \Gamma_j(\cdot) = \sum_{i=0}^{j-1} \alpha_i \ell_f(\cdot;x_i)  + h(\cdot)
	\end{equation}
	where 
  $\{\alpha_0,\ldots,\alpha_{j-1}\} \subset \R_{++}$ are scalars such that
	$\displaystyle \sum_{i=0}^{j-1} \alpha_i =1$.
        \item [(E2)] \textbf{Two-cuts scheme:} 
        For this scheme, it is assumed that
        $\Gamma$
        has the form 
        \begin{equation}\label{def:Gamma-E2}
            \Gamma=\max \{A_f,\ell_f(\cdot;x^-)\}+h
        \end{equation}
        where $h \in \mConv{n}$ and
        $A_f$ is an affine function satisfying $A_f\le f$. In view of \
        $$
        x= \underset{u\in \R^n}\argmin \Gamma (u)=\underset{u\in \R^n}\argmin \max \{A_f(u),\ell_f(u;x^-)\}+h(u),
        $$
        there exists
        $\theta \in [0,1]$ such that
        \begin{align}
            &0 \in    \partial h(x) 
+ \theta \nabla A_f + (1-\theta) f'(x^-), \label{theta1} \\
            &\theta A_f(x) + (1-\theta) \ell_f(x;x^-) = \max \{A_f(x),\ell_f(x;x^-)\}. \label{theta2}
        \end{align}
        The scheme then sets
        \begin{equation}\label{def:Af+}
         A_f^+(\cdot) :=  \theta A_f(\cdot) + (1-\theta) \ell_f(\cdot;x^-)
     \end{equation}
     and outputs the function $\Gamma^+$
      defined as
\begin{equation}\label{eq:G-agg}
		    \Gamma^+(\cdot)  := 
		    \max\{A_f^+(\cdot),\ell_f(\cdot;x)\} + h(\cdot).
	    \end{equation}
    We now show that this model
     $\Gamma^+$ satisfies \eqref{def:Gamma} 
        with $\bar \Gamma = A_f^+ + h$ satisfying
        \eqref{def:bar Gamma}. 
        It is clear that $\Gamma^{+},\bar \Gamma \in \mConv{n}$. Moreover, $A_f(\cdot) \leq f(\cdot)$
        and $\ell_f(\cdot;z) \leq f(\cdot)$ for all 
        $z $ implies that $\Gamma^{+} \leq \phi$. Next, relation
        $\tau \bar \Gamma(\cdot)  + (1-\tau) [\ell_f(\cdot;x)+h(\cdot)] \le \Gamma^+(\cdot)$
        is an immediate consequence
        of the definitions of
        $\bar \Gamma$ and
        $\Gamma^+$ and we have shown that
        $\Gamma^+$ satisfies \eqref{def:Gamma}.
     The relation $\Gamma(x)=\bar \Gamma(x)$
     immediately follows from
     the definition $\bar \Gamma = A_f^+ + h$
     of $\bar \Gamma$ and from relations
     \eqref{def:Gamma-E2}, \eqref{theta2},
     and \eqref{def:Af+}. Finally, by
     \eqref{theta1} we have
     $$
     0 \in \partial h(x) 
+ \theta \nabla A_f + (1-\theta) f'(x^-)
=\partial h(x) + \nabla A_f^{+}=\partial \bar \Gamma(x)
     $$
which can be written $x= \underset{u\in \R^n}\argmin \bar \Gamma (u)$ and we have
shown that $\bar \Gamma$ satisfies \eqref{def:bar Gamma}.    
	    \item [(E3)] {\textbf{Multiple-cuts scheme:} For this scheme, it is assumed that $\Gamma$ is of the form
	    $\Gamma=\Gamma(\cdot;B)$ where
	    $B \subset \R^n$ is a finite set and $\Gamma(\cdot;B)$ is defined by
$$
\Gamma(\cdot;B)=
\max(\ell_f(\cdot;b):b \in B)+h(\cdot).
$$
This scheme 
      chooses the next set $B^+$ so that
	    \begin{equation}\label{eq:C+}
        B(x) \cup \{x\} \subset B^+ \subset B \cup \{x\}
        \end{equation}
        where
        \begin{equation}\label{def:C+}
            B(x) := \{ b \in B : \ell_f(x;b)+h(x) = \Gamma(x) \},
        \end{equation}
         and then
	    outputs $\Gamma^+ = \Gamma(\cdot;B^+)$. We now show that this model
     $\Gamma^+$ satisfies \eqref{def:Gamma} with $\bar \Gamma =\Gamma(\cdot;B(x))$ satisfying
        \eqref{def:bar Gamma}. 
As before, it is clear that $\Gamma^{+},\bar \Gamma \in \mConv{n}$
and that $\Gamma^{+} \leq \phi$
since all linearizations $\ell_f(\cdot;z)$
of $f$ satisfy 
        $\ell_f(\cdot;z) \leq f(\cdot)$ for all 
        $z$. From the definitions of
        $\bar \Gamma$, $\Gamma^{+}$,
        and $B^+$, we have
        $$
        \Gamma^{+}(\cdot) \geq 
        \max(\bar \Gamma(\cdot), \ell_f(\cdot;x)+h),
        $$
which implies 
$\tau \bar \Gamma(\cdot)  + (1-\tau) [\ell_f(\cdot;x)+h(\cdot)] \le \Gamma^+$
and we have shown
that $\Gamma^+$ satisfies
\eqref{def:Gamma}. Next, using the
fact that $\partial \Gamma(x)=\partial  \bar \Gamma(x)$ and the relation
$0 \in \partial  \Gamma(x)$, we deduce
that $x$ minimizes $\bar \Gamma$. Finally,
$$
\bar \Gamma(x)=\max(\ell_f(x;b):b  \in B(x))+h(x)=\max(\ell_f(x;b):b  \in B)+h(x)=\Gamma(x)
$$
and therefore we have shown that $\bar \Gamma$
satisfies \eqref{def:bar Gamma}.

A special case of the multiple-cut
scheme takes $B^{+}=B \cup \{x\}$
and amounts to apply a cutting
plane method for $f$ while
keeping the strongly convex
function $h$ in the objective.

	\end{itemize}

We are now in a position to describe the QCC framework, which is
given below.
	
	\noindent\rule[0.5ex]{1\columnwidth}{1pt}
	
	QCC
	
	\noindent\rule[0.5ex]{1\columnwidth}{1pt}
	\begin{itemize}
		\item [0.] Let $ x_0\in \dom h $, $ \bar \varepsilon>0 $ and $\tau \in (0,1)$ be given
		such that
        \begin{equation}\label{rel:tau1} 
	        \frac{\tau}{1-\tau} \ge \frac{8 T_{\bar \varepsilon}^2}{\mu\bar \varepsilon}
	    \end{equation}
		where $T_{\bar \varepsilon}$ is as in \eqref{def:T}, and set  
		$y_0=x_0$, and $j=1$;
		\item[1.] if $j=1$ then find $\Gamma_{j}$ such that
		\begin{equation}\label{ineq:require}
		    \Gamma_{j}\in \mConv{n}, \quad \ell_f(\cdot;x_{j-1})+h \le \Gamma_{j}\le  \phi;
		\end{equation}
		else, let 
  $\Gamma_{j}=$MU$(\Gamma_{j-1},x_{j-1},\tau)$;
		
	    \item[2.]
	    compute 
		\begin{equation}
	    x_{j} =\underset{u\in \R^n}\argmin \;
	     \Gamma_{j}(u),  \label{def:xj}
	    \end{equation}
		 choose 
   \begin{equation}\label{eq:yj-intuition}
    y_j \in \Argmin \left\lbrace \phi(x) :
		x \in \{x_0,x_1,\ldots,x_j\}
		\right\rbrace, 
\end{equation}
		and set  
		\begin{equation}\label{ineq:hpe1}
		    m_{j}=\Gamma_{j}(x_{j}), \qquad t_{j} = \phi(y_{j}) - m_{j};
		\end{equation}
		\item[3.] set $j\leftarrow j+1$ and go to Step 1.
	\end{itemize}
	\rule[0.5ex]{1\columnwidth}{1pt}








}\fi

At every iteration $j$, QCC computes a new quadratic cut 
$q_{\tilde f}(\cdot,x_{j-1})$ for
$\tilde f$ and
calls the model update function
$QU$ with this cut and the previous model $\Gamma_{j-1}$ for $\tilde f + \tilde h$.
Sequence $t_j$ is
used for a stopping criterion: when $t_j\leq \bar \varepsilon$, we have found an 
$\bar \varepsilon$-optimal solution of
\eqref{eq:ProbIntro1}.\\

\par {\textbf{Application.}} Consider the problem
\begin{equation}\label{pboptdef}
\min \; \{\tilde f(x):x \in X\}
\end{equation}
with $\tilde f$ $\mu$-convex
and $X \subset \mathbb{R}^n$ convex, compact, and nonempty.

Function $\tilde f$ being $\mu$-convex, recall that 
it is lower bounded by quadratic
functions, namely \eqref{quadminorant}
holds.

Therefore, two natural models for $\tilde f$ in this situation are a maximum of lower bounding quadratic cuts
and a linear combination of lower bounding
quadratic cuts. This gives rise to two algorithms computing
$x_j$ using the iterations
\begin{equation}\label{xjcompute}
x_j = \displaystyle \mbox{argmin}_{u \in \mathbb{R}^n} \; \Gamma_j(u)
\end{equation}
where for the first algorithm
the model $\Gamma_j$ is
\begin{equation}\label{boundq1}
\Gamma_j(u)=\max(q_{{\tilde f}}(u;x_0),\ldots,q_{\tilde f}(u;x_{j-1}))+\tilde h(u)
\end{equation}
while for the second algorithm
the model $\Gamma_j$ is given by
\begin{equation}\label{boundq2}
\Gamma_j(u)=\sum_{i=0}^{j-1} \alpha_i q_{\tilde f}(u;x_{i})+\tilde h(u)
\end{equation}
where $\tilde h=\delta_X$ is the indicator function of set $X$ and is convex.
While these algorithms are quite natural
for solving problem \eqref{pboptdef},
we are not aware of another paper
proposing these algorithms.
However, it can be seen that
these algorithms are special
instances of QCC framework.
Therefore, the complexity of these
algorithms is also given
by Theorem \ref{complcssc} below which gives the complexity of QCC.

	\subsection{Complexity analysis of QCC}\label{sec:analysis}

\subsubsection{Reformulation of QCC}\label{reformQCC}

 In this section, we provide a reformulation of
 QCC in terms of linear (instead of quadratic in the
 original QCC formulation) cuts.
The model update black-box QU$(\Gamma,x,\tau)$
defines a new model $\Gamma^{+}$ satisfying
\begin{equation}\label{reformmodu}
	    \Gamma^+ \in \mConv{n}, \qquad \tau \bar \Gamma(\cdot)  + (1-\tau) [q_{\tilde f}(\cdot;x)+{\tilde h}(\cdot)] \le \Gamma^+(\cdot) \le \phi(\cdot).
\end{equation}
Define functions $f$ and
$h$ by
\begin{equation}\label{defftildehtilde}
f(x)=\tilde f(x)-\frac{\mu}{2}\|x\|^2 \mbox{ and }h(x)=\tilde h(x)+\frac{\mu}{2}\|x\|^2.
\end{equation}
Observe that problem \eqref{eq:ProbIntro1}
can be written
\begin{equation}\label{reformcomposite}
\min_{x \in \mathbb{R}^n} \;f(x) + h(x).
\end{equation}
Since $\tilde h$ is convex we have that
$h$
is $\mu$-strongly convex. We also have that
$\tilde f \in \mConv{n}$
is $\mu$-convex
and therefore
$f$ is convex.

Note that
there is $\tilde f'(x) \in \partial \tilde f(x)$
and 
$f'(x) \in \partial f(x)$
such that 
$f'(x)=\tilde f'(x)-\mu x$ and we can write
\begin{equation}\label{reformlincut}
\begin{array}{lcl}
\ell_f(u;x)+h(u) & = & f(x) + \langle f'(x),u-x\rangle +h(u) \\
& = & f(x) + \langle {\tilde f}'(x)-\mu x,u-x\rangle +\tilde h(u) +\frac{\mu}{2}\|u\|^2\\
& = & \tilde f(x)-\frac{\mu}{2}\|x\|^2 + \langle {\tilde f}'(x),u-x\rangle - \mu\langle x, u-x \rangle +\tilde h(u) +\frac{\mu}{2}\|u\|^2\\
&=& q_{\tilde f}(u;x)+\tilde h(u),
\end{array}
\end{equation}
which shows that the model written in terms
of quadratic cuts $q_{\tilde f}(u;x)$
can be written in terms
of linear cuts 
$\ell_f(u;x)$ 
written for $f$ instead
of $\tilde f$ but also
replacing $\tilde h$ by $h$, which is now $\mu$-convex. It follows  that
the QB update function
can be reformulated replacing
the inequalities in
\eqref{def:Gamma}
by 
$$
\tau \bar \Gamma(\cdot)  + (1-\tau) [\ell_f(u;x)+h(u)] \le \Gamma^+(\cdot) \le \phi(\cdot),
$$
with corresponding QCC method.
We see that we have transferred
the $\mu$-convexity from $\tilde f$ to
$h$ 
and the convexity of $\tilde h$
to $f$ for which linear cuts
can be computed.
Recall that 
the objective of composite
optimization problem \eqref{eq:ProbIntro1} is the sum
of a convex function $\tilde f$ which is
approximated in QCC  by a model 
constructed from quadratic cuts
$q_{\tilde f}(\cdot;x)$ of $\tilde f$
and of convex function $\tilde h$.
This problem can also be expressed
as composite optimization problem
\eqref{reformcomposite}
with objective which is
the sum of a $\mu$-convex
function $h$ and
convex function $f$.

From our previous observations, for iteration $j=1$, we can replace
condition \eqref{ineq:require} 
by
		$$
		    \Gamma_{1}\in \bConv{n} \mbox{ and } \ell_f(\cdot;x_{0})+h \le \Gamma_{1}\le  \phi.
		$$

Using once again \eqref{reformlincut},
for an iteration $j \geq 2$, the model
  $\Gamma_j$ in QCC can  also be obtained using a
  Linear Update function LU$(\Gamma_{j-1},x_{j-1},\tau)$ given below, depending
  on three inputs: the previous model
  $\Gamma_{j-1}$, the previous approximate
  solution $x_{j-1}$, and a parameter
  $\tau \in (0,1)$. 
\\
\noindent\rule[0.5ex]{1\columnwidth}{1pt}
	
	LU$(\Gamma,x,\tau)$

    \noindent\rule[0.5ex]{1\columnwidth}{1pt}
    {\bf Input:} Let $\tau \in (0,1)$ and
    $(x,\Gamma) \in \R^n \times \bConv{n}$ such that $\Gamma\le \phi$.

\noindent Find function $\Gamma^+$ such that
\begin{equation}\label{def:Gammal}
	    \Gamma^+ \in \bConv{n}, \qquad \tau \bar \Gamma(\cdot)  + (1-\tau) [\ell_f(\cdot;x)+h(\cdot)] \le \Gamma^+(\cdot) \le \phi(\cdot),
	\end{equation}
   and where $\bar \Gamma(\cdot) $ is such that
\begin{equation}\label{def:bar Gammal}
	\bar \Gamma \in \bConv{n}, \quad
 \bar \Gamma(x) = \Gamma(x), \quad 
	x = \underset{u\in \R^n}\argmin \;\bar \Gamma (u).
	\end{equation}
     {\bf Output:} $\Gamma^+$.
     
    \noindent\rule[0.5ex]{1\columnwidth}{1pt}

	We now describe three concrete
 update schemes (L1), (L2), and (L3) which are special ways
 of implementing LU function. 
 They are the analogues of Examples (E1), (E2), and
 (E3) based on quadratic cuts.
	\begin{itemize} 
     \item [(L1)] {\bf One-cut scheme}: 
    	    This scheme obtains $\Gamma^+$ as
	    \begin{equation}\label{eq:affine1}
        \Gamma^+= \Gamma^+_\tau := \tau \Gamma + (1-\tau) [\ell_f(\cdot;x)+h]
        \end{equation}
        where $x = \underset{u\in \R^n}\argmin \;\Gamma (u)$ and $\tau\in (0,1)$.  Clearly, $\Gamma^+$ satisfies 
        \eqref{def:Gammal} 
        with $\bar \Gamma = \Gamma$ satisfying
        \eqref{def:bar Gammal}. 
    Also, if
        $\Gamma$ is the sum of $h$ and
        an affine function underneath $f$,
        then so is $\Gamma^+$.	   
Such is the case if the first model 
$\Gamma_1$ is chosen to be 
$\ell_f(\cdot;x_0)+h$
     in which case
     $\Gamma_j$ is of the form
	\begin{equation}\label{eq:Gamma-form1}
	    \Gamma_j(\cdot) = \sum_{i=0}^{j-1} \alpha_i \ell_f(\cdot;x_i)  + h(\cdot)
	\end{equation}
	where 
  $\{\alpha_0,\ldots,\alpha_{j-1}\} \subset \R_{++}$ are scalars such that
	$\displaystyle \sum_{i=0}^{j-1} \alpha_i =1$.
        \item [(L2)] \textbf{Two-cuts scheme:} 
        For this scheme, it is assumed that
        $\Gamma$
        has the form 
        \begin{equation}\label{def:Gamma-E2l}
            \Gamma=\max \{A_f,\ell_f(\cdot;x^-)\}+h
        \end{equation}
        where $h \in \mConv{n}$ and
        $A_f$ is an affine function satisfying $A_f\le f$. In view of \
        $$
        x= \underset{u\in \R^n}\argmin \;\Gamma (u)=\underset{u\in \R^n}\argmin \max \{A_f(u),\ell_f(u;x^-)\}+h(u),
        $$
        there exists
        $\theta \in [0,1]$ such that
        \begin{align}
            &0 \in    \partial h(x) 
+ \theta \nabla A_f + (1-\theta) f'(x^-), \label{theta1l} \\
            &\theta A_f(x) + (1-\theta) \ell_f(x;x^-) = \max \{A_f(x),\ell_f(x;x^-)\}. \label{theta2l}
        \end{align}
        The scheme then sets
        \begin{equation}\label{def:Af+l}
         A_f^+(\cdot) :=  \theta A_f(\cdot) + (1-\theta) \ell_f(\cdot;x^-)
     \end{equation}
     and outputs the function $\Gamma^+$
      defined as
\begin{equation}\label{eq:G-agg2}
		    \Gamma^+(\cdot)  := 
		    \max\{A_f^+(\cdot),\ell_f(\cdot;x)\} + h(\cdot).
	    \end{equation}
     Same as for Example (E2), we can show that this model
     $\Gamma^+$ satisfies \eqref{def:Gammal} 
        with $\bar \Gamma = A_f^+ + h$ satisfying
        \eqref{def:bar Gammal}. 
        \item [(L3)] {\textbf{Multiple-cuts scheme:}} For this scheme, it is assumed that $\Gamma$ is of the form
	    $\Gamma=\Gamma(\cdot;B)$ where
	    $B \subset \R^n$ is a finite set and $\Gamma(\cdot;B)$ is defined by
$$
\Gamma(\cdot;B)=
\max(\ell_f(\cdot;b):b \in B)+h(\cdot).
$$
This scheme 
      chooses the next set $B^+$ so that
	    \begin{equation}\label{eq:C+l}
        B(x) \cup \{x\} \subset B^+ \subset B \cup \{x\}
        \end{equation}
        where
        \begin{equation}\label{def:C+l}
            B(x) := \{ b \in B : \ell_f(x;b)+h(x) = \Gamma(x) \},
        \end{equation}
         and then
	    outputs $\Gamma^+ = \Gamma(\cdot;B^+)$. Same as Example
     (E3), we can  show that this model
     $\Gamma^+$ satisfies \eqref{def:Gammal} with $\bar \Gamma =\Gamma(\cdot;B(x))$ satisfying
        \eqref{def:bar Gammal}. 
        A special case of the multiple-cut
scheme takes $B^{+}=B \cup \{x\}$
and amounts to apply a cutting
plane method for $f$ while
keeping the strongly convex
function $h$ in the objective.

	\end{itemize}

We are now in a position to describe a reformulation of the QCC framework.
	
	\noindent\rule[0.5ex]{1\columnwidth}{1pt}
	
	QCC based on model update function LU (instead of QU)
	
	\noindent\rule[0.5ex]{1\columnwidth}{1pt}
	\begin{itemize}
		\item [0.] Let $ x_0\in \dom h $, $ \bar \varepsilon>0 $ and $\tau \in (0,1)$ be given.
Set  $y_0=x_0$ and $j=1$.
		\item[1.] if $j=1$ then find $\Gamma_{j}$ such that
		\begin{equation}\label{ineq:requirel}
		    \Gamma_{j}\in \bConv{n}, \quad \ell_f(\cdot;x_{j-1})+h \le \Gamma_{j}\le  \phi;
		\end{equation}
		else, let 
  $\Gamma_{j}=$LU$(\Gamma_{j-1},x_{j-1},\tau)$;
		
	    \item[2.]
	    compute 
		\begin{equation}
	    x_{j} =\underset{u\in \R^n}\argmin \;
	     \Gamma_{j}(u),  \label{def:xjl}
	    \end{equation}
		 choose 
   \begin{equation}\label{eq:yj-intuitionl}
    y_j \in \Argmin \left\lbrace \phi(x) :
		x \in \{x_0,x_1,\ldots,x_j\}
		\right\rbrace, 
\end{equation}
		and set  
		\begin{equation}\label{ineq:hpe1l}
		    m_{j}=\Gamma_{j}(x_{j}), \qquad t_{j} = \phi(y_{j}) - m_{j};
		\end{equation}
		\item[3.] set $j\leftarrow j+1$ and go to Step 1.
	\end{itemize}
	\rule[0.5ex]{1\columnwidth}{1pt}








\subsubsection{Complexity analysis}

Our complexity analysis of QCC 
is based on the reformulation of
QCC given in Section \ref{reformQCC}
that uses model update function LU
and functions $f$ and $h$ given by
\eqref{defftildehtilde}.

We assume that the following conditions
hold
	for some triple
	$(L_f, M_f,\mu) \in \R_+^3 $:
	\begin{itemize}
		\item[(A1)]
		$f \in \bConv{n}$ and $h \in \mConv{n}$ are such that
		$\dom h \subset \dom f$, and a subgradient oracle, 		 i.e.,
		a function $f':\dom h \to \R^n$
		satisfying $f'(x) \in \partial f(x)$ for every $x \in \dom h$, is available;
		\item[(A2)]
		the set of optimal solutions $X^*$ of
		problem \eqref{eq:ProbIntro1} is nonempty;
		\item[(A3)]
		for every $u,v \in \dom h$,
		\begin{equation}
		\|f'(u)-f'(v)\| \le 2M_f + L_f \|u-v\|.
		\end{equation}
	\end{itemize}

We now add a few remarks about assumptions (A1)-(A3).
	First, the set $\Omega \subset \R^2_+$ consisting of
	 the pairs $(M_f,L_f)$
	satisfying (A3) is a nonempty closed convex set. Moreover,
	for a given tolerance
	$\bar \varepsilon>0$, 
	there exists a unique
	pair $(\bar M_f(\bar \varepsilon),\bar L_f(\bar \varepsilon))$ which minimizes
	$M_f^2+\bar \varepsilon L_f$ over $\Omega$ and,
	without any loss of clarity,
	we denote this pair simply by
	$(\bar M_f,\bar L_f)$ and
	define
		\begin{equation}\label{def:T}
	    T_{\bar \varepsilon} := \left( \bar M_f^2+\bar \varepsilon \bar L_f \right)^{1/2}.
	\end{equation}
    Second, it is well-known that (A3) implies that for every $u,v \in \dom h$,
	\begin{equation}\label{ineq:est}
	f(u)-\ell_f(u;v) \le 2M_f \|u-v\| + \frac{L_f}{2}\|u-v\|^2.
	\end{equation}

In addition to the above quantities, we introduce
the
	distance of a given initial
	point $x_0$ to $X^*$:
\begin{equation}\label{def:d0}
	d_0 := \|x_0-x_0^*\|, \ \ \mbox{\rm where} \ \ \ 
	x_0^* := \argmin \{\|x_0-x^*\|: x^*\in X^*\}.
\end{equation}

		 Finally,
for given initial point $x_0 \in \dom h$
and tolerance $\bar \varepsilon>0$, it is said that an algorithm for solving \eqref{eq:ProbIntro1}
has $\bar \varepsilon$-iteration complexity ${\cal O}(N)$ if its total number of iterations until it obtains a $ \bar \varepsilon $-solution
 is bounded by 
 $C(N+1)$ where $C>0$ is a
universal constant. 
 
		The first result below describes some
 basic properties of a
 sequence of auxiliary functions $\{\bar \Gamma_j\}$ whose existence is guaranteed by the nature of the LU blackbox.
	
	\begin{lemma}\label{lem:101}
	For every $j \geq 1$, the following statements hold:
	    \begin{itemize}
	        \item[a)] there exists function $\bar \Gamma_{j}(\cdot)$ such that
	        \begin{align}
	            &\tau \bar \Gamma_{j}(\cdot) + (1-\tau) [\ell_f(\cdot;x_j)+h(\cdot)] \le \Gamma_{j+1}(\cdot) \le \phi(\cdot), \label{eq:Gamma_j} \\
	            &\bar \Gamma_{j} \in \mConv{n}, \quad \bar \Gamma_{j}(x_j) = \Gamma_{j}(x_j), \quad 
	        x_j = \underset{u\in \R^n}\argmin \; \bar \Gamma_{j} (u); \label{eq:relation}
	        \end{align} 
	        \item[b)] for every $u\in \R^n$, we have
       \begin{equation}\label{ineq:Gammaj}
           \bar \Gamma_{j}(u) \ge m_j+ \frac{\mu}{2}\|u-x_j\|^2 .
       \end{equation}
	    \end{itemize}
	\end{lemma}
	\begin{proof}
	    a) This statement immediately follows from \eqref{def:Gammal}, \eqref{def:bar Gammal}, and the facts that $\Gamma_{j+1}$ is the output of the MU black-box with input $(\Gamma_{j},x_{j},\tau)$.
	    
	    b) It follows from $\bar \Gamma_{j}\in \mConv{n}$
 that $\bar \Gamma_{j}$ is $\mu$-strongly convex. 
 Thus, for every $u\in\dom h$,
	    $\bar \Gamma_{j}(u)  \ge \bar \Gamma_{j}(x_j) + \frac{\mu}{2}\|u-x_j\|^2$.
     Then, using the second identity in \eqref{eq:relation} and the definition of $m_j$ in~\eqref{ineq:hpe1l}, we get 
     $\bar \Gamma_{j}(x_j) = \Gamma_{j}(x_j)= m_j$ which ends the proof.
	\end{proof}
	
	\par The following technical result provides an important recursive formula for sequence $\{m_j\}$
	which is used in Lemma \ref{lem:tj} to give
	a recursive formula for 
 sequence
 $\{t_j\}$.
	It is worth observing that its proof
	uses for the first time the condition
	 \eqref{rel:tau1}.

	\begin{lemma}\label{lem:recur}
		Suppose $\tau$ in QCC satisfies
         \begin{equation}\label{rel:tau1} 
	        \frac{\tau}{1-\tau} \ge \frac{8 T_{\bar \varepsilon}^2}{\mu\bar \varepsilon}
	    \end{equation}
		where $T_{\bar \varepsilon}$ is as in \eqref{def:T}.
  Then for every 
  $j \geq 1$, 
        we have 
\begin{equation}\label{ineq:mj1}
		    m_{j+1} \ge \tau m_j + (1-\tau) \left[ \ell_f(x_{j+1};x_j) + h(x_{j+1}) + \left(\frac{\bar L_f}2+\frac{4\bar M_f^2}{\bar \varepsilon}\right) \|x_{j+1}-x_j\|^2 \right].
		\end{equation}
	\end{lemma}
    
	\begin{proof}
First, it immediately follows from the assumption over $\tau$ in~\eqref{rel:tau1} and the definition of $T_{\bar \varepsilon}$  in~\eqref{def:T}, that	    \begin{equation}\label{rel:tau2} 
	        \frac{\tau}{1-\tau} \ge \frac{8(\bar M_f^2 + \bar \varepsilon \bar L_f)}{\mu \bar \varepsilon} \ge \frac{1}{\mu} \left(\bar L_f + \frac{8\bar M_f^2}{\bar \varepsilon}\right).
	    \end{equation}
		Using the lower-bound hypothesis on $\Gamma_j$~\eqref{eq:Gamma_j}, the fact that $\tau<1$, and \eqref{ineq:Gammaj} with $u=x_{j+1}$, we have
		\begin{align*}
		m_{j+1} & \overset{\eqref{ineq:hpe1l}}{=} \Gamma_{j+1}(x_{j+1}) \\
		&\overset{\eqref{eq:Gamma_j}}{\ge} (1-\tau) [\ell_f(x_{j+1};x_j) + h(x_{j+1})] + \tau  \bar \Gamma_{j}(x_{j+1}) \\
		& \overset{\eqref{ineq:Gammaj}}{\ge} (1-\tau) [\ell_f(x_{j+1};x_j) + h(x_{j+1})] + \tau \left( m_j + \frac{\mu}{2} \|x_{j+1} -x_j\|^2 \right)\\
        & =  \tau m_j + (1-\tau)  
        \left[ \ell_f(x_{j+1};x_j) + h(x_{j+1}) + \frac{\tau}{1-\tau} \frac{\mu}{2} \|x_{j+1} -x_j\|^2\right]\\
        & \overset{\eqref{rel:tau2}}{\ge} \tau m_j + (1-\tau)  
        \left[ \ell_f(x_{j+1};x_j) + h(x_{j+1}) + \left(\frac{\bar L_f}2+\frac{4\bar M_f^2}{\bar \varepsilon}\right) \|x_{j+1} -x_j\|^2\right]
		\end{align*}
		which ends the proof.
	\end{proof}
\begin{rem} We could improve
the right-hand side of \eqref{ineq:mj1}
replacing $\frac{\bar L_f}2$ by
$4 \bar L_f$. We forced the term
$\frac{\bar L_f}2$ instead in the right-hand side since it will appear naturally
in subsequent computations. 
\end{rem}
 
	The next result,
 which
	plays an important role in the analysis,
 establishes
	a key recursive formula for the sequence $\{t_j\}$ defined in \eqref{ineq:hpe1l}.
	
	\begin{lemma}\label{lem:tj}
	    For every  $j \geq 1$, we have
\begin{equation}\label{ineq:tj-recur}
         t_{j+1}-\frac{\bar \varepsilon}4 \le \tau \left(t_j -\frac{\bar \varepsilon}4\right).
     \end{equation}
	\end{lemma}
	
	\begin{proof} 
        Using \eqref{ineq:est} with $(M_f,L_f,u,v)=(\bar M_f, \bar L_f, x_{j+1},x_j)$ we get 
        \begin{equation} \ell_f(x_{j+1};x_j)  + \frac{\bar L_f}2 \|x_{j+1}-x_j\|^2 + 2 \bar M_f \|x_{j+1}-x_j\| \ge  f(x_{j+1}), \end{equation}
        and the fact that $\phi=f+h$,  yields
        \begin{equation}\label{ineq:ellphi}
            \ell_f(x_{j+1};x_j) + h(x_{j+1}) + \frac{\bar L_f}2 \|x_{j+1}-x_j\|^2\ge \phi(x_{j+1}) - 2 \bar M_f \|x_{j+1}-x_j\|.
        \end{equation}
        This inequality and \eqref{ineq:mj1} imply that
		\begin{align}
			m_{j+1} - \tau m_j
			&\overset{\eqref{ineq:mj1}}{\ge} (1-\tau) \left[ \ell_f(x_{j+1};x_j) + h(x_{j+1}) +   \left(\frac{\bar L_f}2 +\frac{4\bar M_f^2}{\bar \varepsilon}\right) \|x_{j+1}-x_j\|^2 \right] \nn \\
			&\overset{\eqref{ineq:ellphi}}{\ge} (1-\tau) \left[ \phi(x_{j+1}) - 2 \bar M_f \|x_{j+1}-x_j\| +  \frac{4\bar M_f^2}{\bar \varepsilon} \|x_{j+1}-x_j\|^2 \right] \nn\\
   & = (1-\tau) \phi(x_{j+1}) + \frac{1-\tau}{\bar \varepsilon}\left(4\bar M_f^2\|x_{j+1} -x_j\|^2 - 2\bar M_f \bar \varepsilon\|x_{j+1} -x_j\|\right) \nn \\
			&\ge   (1-\tau) \phi(x_{j+1}) -  \frac{(1-\tau)\bar \varepsilon}{4}, \label{ineq:mj2}
		\end{align}		
		where the last inequality is due to the inequality
		$ a^2-2ab \ge - b^2$ with $a=2\bar M_f\|x_{j+1}-x_j\|$ and $b=\bar \varepsilon/2$.
		Using the above inequality and the definitions of $y_{j+1}$ and $ t_{j+1} $ in \eqref{eq:yj-intuitionl} and \eqref{ineq:hpe1l}, respectively, we conclude that
        \begin{align*}
            t_{j+1} &\overset{\eqref{ineq:hpe1l}}{=} \phi(y_{j+1}) - m_{j+1} \overset{\eqref{ineq:mj2}}{\le} \phi(y_{j+1}) -\tau m_j - (1-\tau) \phi(x_{j+1}) + \frac{(1-\tau)\bar \varepsilon}{4} \\
            &\overset{\eqref{ineq:hpe1l}}{=} \phi(y_{j+1}) -\tau [\phi(y_j)-t_j] - (1-\tau) \phi(x_{j+1}) + \frac{(1-\tau)\bar \varepsilon}{4} \\
            &\overset{\eqref{eq:yj-intuitionl}}{\le} \tau t_j + \frac{(1-\tau)\bar \varepsilon}{4},
        \end{align*}
  and that the lemma holds.
	\end{proof}
	

	

	The next lemma gives a uniform bound on $t_{1}$.
    \begin{lemma}\label{lem:t1} 
    Assume that the domain of $h$
    is bounded with diameter at most 
    $D$. Then we have $ t_{1}\le \bar t$
		where	\begin{equation}\label{def:bar t}
		    \bar t:= \bar M_f^2 + 
       \left( \frac{{\bar L}_f}{2} +1\right) D^2.
		\end{equation}
	\end{lemma}
	
	\begin{proof}
	    Using relation $\phi=f+h$ and $ \Gamma_{1}\ge \ell_f(\cdot;x_{0})+h $, we have
	    \begin{align*}
	        t_{1} &\overset{\eqref{ineq:hpe1l}}{=} \phi(y_{1}) - m_{1}
         \leq 
\phi(x_{1}) - \Gamma_{1}(x_{1})
	        \le f(x_{1}) - \ell_f(x_{1};x_{0}) \\
	        & \le  2\bar M_f \|x_{1}-x_{0}\| + \frac{\bar L_f}{2} \|x_{1}-x_{0}\|^2
	        \le \bar M_f^2 + \left( \frac{\bar L_f}{2}+1\right)\|x_{1}-x_{0}\|^2
	    \end{align*}
		where the third inequality is due to \eqref{ineq:est} with $(M_f, L_f, u,v)=(\bar M_f, \bar L_f, x_{1}, x_{0})$, and the last inequality is due to the fact that $2ab \le a^2+b^2$ for every $a,b \in \R$.
	\end{proof}

 The next theorem gives the
 complexity of QCC.
    	\begin{theorem}[Complexity of QCC]\label{complcssc} 
Consider QCC framework with $\tau$ given
by 
\begin{equation}\label{partchtau}
\frac{1}{\tau}=1+\frac{\mu \bar \varepsilon}{8 T^2_{\bar \varepsilon}}=1+\frac{\mu \bar \varepsilon}{8(\bar M_f^2+\bar \varepsilon \bar L_f)}. 
\end{equation}
    Assume that the domain of $h$
    is bounded with diameter $D>0$. 
    Then if 
\begin{equation}\label{condj}
j \geq 1 + \left[1+ \frac{8\left( \bar M_f^2+\bar \varepsilon \bar L_f \right)}{\mu \bar \varepsilon}
\right]\log \left( \frac{4 \bar t}{3 \bar \varepsilon}  \right)
\end{equation}
where $\bar t$ is given in Lemma \ref{lem:t1},
we have $t_j \leq \bar \varepsilon$ and
QCC finds an $\bar \varepsilon$-optimal solution
of \eqref{eq:ProbIntro1} in at most 
$$
1 + \left[1+ \frac{8\left( \bar M_f^2+\bar \varepsilon \bar L_f \right)  }{\mu \bar \varepsilon}
\right]\log \left( \frac{4 \bar t}{3 \bar \varepsilon}  \right)
$$
iterations.
	\end{theorem}
	\begin{proof}
 Using Lemma \ref{lem:tj}
and the relation $\tau \leq e^{\tau-1}$, we get
for every $j \geq 1$:

\begin{align}
    t_j  & \leq  \frac{\bar \varepsilon}{4} +\tau^{j-1}(t_1-\frac{\bar \varepsilon}{4})
    & \text{induction on \eqref{ineq:tj-recur}} \nonumber\\
 & \leq \frac{\bar \varepsilon}{4} +\tau^{j-1} \bar t & \text{by Lemma \ref{lem:t1}} \nonumber\\
 & \leq \frac{\bar \varepsilon}{4}
 +e^{(\tau-1)(j-1)} \bar t & \text{since } \tau \leq e^{\tau-1}.\label{tjfinal}
\end{align}
Further, 
$$ \tau-1 = \frac{- \frac{\mu \bar \varepsilon}{8T^2_{\bar \varepsilon}}}{1+\frac{\mu \bar \varepsilon}{8T^2_{\bar \varepsilon}}}
= - \frac{1}{1+\frac{8T^2_{\bar \varepsilon}}{\mu \bar \varepsilon}}.$$
So, for $j \geq 1+\left(1 + \frac{8T^2_{\bar \varepsilon}}{\mu \bar \varepsilon}\right)\log(4\bar t / (3 \bar \varepsilon))$, we have 
$
(j-1)(\tau-1) \leq \log(3 \bar \varepsilon/4 \bar t)
$ or equivalently
$\exp((\tau-1)(j-1)) \leq 3 \bar \varepsilon/(4\bar t)$, which plugged into
\eqref{tjfinal} gives
$ t_j \leq \frac{\bar \varepsilon}{4} + \frac{3\bar \varepsilon}{4} = \bar \varepsilon$.

\if{
\oldvl{
 $$
 t_j \leq \frac{\varepsilon}{4}
 +\tau^{j-1}(t_1-\frac{\varepsilon}{4})
 \leq \frac{\varepsilon}{4}
 +\tau^{j-1} \bar t
 \leq \frac{\varepsilon}{4}
 +e^{(\tau-1)(j-1)} \bar t 
 \leq \varepsilon
 $$
where the last inequality follows from
relation \eqref{condj}
and the choice of $\tau$
given by \eqref{partchtau}.}

}\fi

Finally,
$$
0 \leq \phi(y_j)-\phi_* \leq 
\phi(y_j)-\Gamma_j(x_j) =t_j \leq \bar \varepsilon.
$$
and $y_j$ is an $\bar \varepsilon$-optimal solution.
\end{proof}

\if{

\section{QCC in terms of quadratic cuts and application}

 In this section, we provide a reformulation of
 QCC in terms of linear (instead of quadratic in the
 original QCC formulation) cuts and discuss an application.
Define $f$ and $h$ by the relations
\begin{equation}\label{deffh}
f(x)=\tilde f(x)-\frac{\mu}{2}\|x\|^2 \mbox{ and }
h(x)=\tilde h(x)+\frac{\mu}{2}\|x\|^2.
\end{equation}

Observe that problem \eqref{eq:ProbIntro1}
can be written
\begin{equation}\label{reformcomposite}
\min_{x \in \mathbb{R}^n} \;f(x) + h(x).
\end{equation}
Since $\tilde f$ is $\mu$-convex we have that
$f \in  {\overline{\mbox{Conv}}(\mathbb{R}^n)}$
is convex. We also have that
$h \in \mConv{n}$
is $\mu$-convex. Note that
there is $\tilde f'(x) \in \partial \tilde f(x)$
such that $\tilde f'(x)=f'(x)+\mu x$ and
recall that the quadratic
approximation 
$$
q_{\tilde f}(u;x) =\tilde f(x)
+\langle \tilde f'(x),u-x\rangle + \frac{\mu}{2}\|u-x\|^2
$$
of $\tilde f$ at $x$ satisfies 
\begin{equation}\label{quadminorant}
\tilde f(u) \geq q_{\tilde f}(u;x)
\end{equation}
for every $u \in \mathbb{R}^n$.
Next, the term $\ell_f(u;x)+h(u)$
in \eqref{reformmodu} can be written
\begin{equation}\label{reformlincut}
\begin{array}{lcl}
\ell_f(u;x)+h(u) & = & f(x) + \langle f'(x),u-x\rangle +h(u) \\
& = & f(x) + \langle {\tilde f}'(x)-\mu x,u-x\rangle +\tilde h(u) +\frac{\mu}{2}\|u\|^2\\
& = & \tilde f(x)-\frac{\mu}{2}\|x\|^2 + \langle {\tilde f}'(x),u-x\rangle - \mu\langle x, u-x \rangle +\tilde h(u) +\frac{\mu}{2}\|u\|^2\\
&=& q_{\tilde f}(u;x)+\tilde h(u),
\end{array}
\end{equation}
which shows that the model can be written in terms
of quadratic cuts $q_{\tilde f}(u;x)$
computed at $x$ and  replacing $h$ by
$\tilde h$. It follows  that
the MU model update blackbox
can be reformulated replacing
the inequalities in
\eqref{def:Gamma}
by 
$$
\tau \bar \Gamma(\cdot)  + (1-\tau) [q_{\tilde f}(u;x)+\tilde h(u)] \le \Gamma^+(\cdot) \le \phi(\cdot),
$$
with corresponding QCC method.
We see that we have transferred
the $\mu$-convexity from $h$ to
$\tilde f$ for which quadratic cuts can be computed. Recall that 
the objective of composite
optimization problem \eqref{eq:ProbIntro1} is the sum
of a $\mu$-convex function $f$ which is
approximated in QCC  by a model $\Gamma_j$
constructed from quadratic cuts
$q_{\tilde f}(\cdot;x)$ of $\tilde f$
and of convex function $\tilde h$.
This problem can also be expressed
as composite optimization problem
\eqref{reformcomposite}
with objective which is
the sum of a $\mu$-convex
function $\tilde f$ and
convex function $\tilde h$.
Therefore, QCC
could also in theory be applied
building a model for
$\tilde h$ (the convex, not necessarily $\mu$-convex for some $\mu>0$) and keeping $\tilde f$ (the $\mu$-convex function)
as it is. However, in practice,
a model will be built for the
"most complicated" function.

We can now express the three particular
models (E1), (E2), and (E3) from
Section \ref{subsec:update} in terms
of quadratic cuts. 

	\begin{itemize} 
     \item [(E1)] {\bf One-cut scheme}:
     From \eqref{reformlincut}, we
     see that one-cut scheme (E1)
   obtained taking at the first
   iteration model 
$\Gamma_1$ given by 
$q_{\tilde f}(\cdot;x_0)+\tilde h$
    provides at iteration $j$
    of QCC a model
     $\Gamma_j$ of the form
	\begin{equation}\label{eq:Gamma-formrewrite}
	    \Gamma_j(\cdot) = \sum_{i=0}^{j-1} \alpha_i q_{\tilde f}(\cdot;x_i)  + {\tilde h}(\cdot)
	\end{equation}
which is $\tilde h$ plus a quadratic
model which is a convex combination of
the quadratic cuts computed up to
iteration $j$.
        \item [(E2)] \textbf{Two-cuts scheme:} For
        scheme (E2), model $\Gamma$
        has the form 
        \begin{equation}\label{def:Gamma-E2rew}
            \Gamma=\max \{Q_f,q_{\tilde f}(\cdot;x^-)\}+{\tilde h}
        \end{equation}
        where 
        $Q_f$ is a quadratic function satisfying $Q_f\le f$.
 \item [(E3)] {\textbf{Multiple-cuts scheme:} For this scheme, $\Gamma$ is of the form
	    $\Gamma=\Gamma(\cdot;B)$ where
	    $B \subset \R^n$ is a finite set and $\Gamma(\cdot;B)$ is defined by
$$
\Gamma(\cdot;B)=
\max(q_{\tilde f}(\cdot;b):b \in B)+{\tilde h}(\cdot).
$$
This scheme 
      chooses the next set $B^+$ so that
	    \begin{equation}\label{eq:C+}
        B(x) \cup \{x\} \subset B^+ \subset B \cup \{x\}
        \end{equation}
        where
        \begin{equation}\label{def:C+}
            B(x) := \{ b \in B : q_{\tilde f}(x;b)+{\tilde h}(x) = \Gamma(x) \}.
        \end{equation}
\end{itemize}	

}\fi

\section{Bundle methods based on quadratic cuts}\label{sec:qb}

In section, we are still interested in solving problem \eqref{eq:ProbIntro1}.
We consider an extension of the bundle methods from
\cite{montliang23} where linearizations of the 
"complicate" function in the objective are replaced
by quadratic cuts. We call QB (Quadratic Bundle) the corresponding
method.
An iteration of QB solves the prox bundle problem
\begin{equation}\label{defx}
x =  \argmin_{u \in \mathbb{R}^n}\;
\Gamma(u)+\frac{1}{2 \lambda}\|u-x^c\|^2
\end{equation}
where $\lambda$ is the prox stepsize, $x^c$ is the prox-center,
and $\Gamma$ is a model for $\tilde f + \tilde h$.
The method has two types of iterations: serious and null ones. In a serious
iteration, the prox-center is updated to $x^c \leftarrow x$
and the updated bundle function $\Gamma^{+}$
satisfies $\Gamma^{+} \geq q_{\tilde f}(;x)+\tilde h$
where $q_{\tilde f}(;x)$
is the quadratic approximation
of $\tilde f$ given by
\eqref{quadapprox}.

In a null iteration, the prox-center
does not change and $\Gamma$ is updated 
using some bundle update procedure.

The method is based on the Quadratic Bundle Update (QBU) blackbox
given below, depending on a parameter $\tau \in (0,1)$
and prox-step size $\lambda$.

\noindent\rule[0.5ex]{1\columnwidth}{1pt}
	
 QBU$(\Gamma,x,\tau,x^c)$

    \noindent\rule[0.5ex]{1\columnwidth}{1pt}
    {\bf Input:} Let $\lambda>0$, $\tau \in (0,1)$,
    $x^c, x \in \mathbb{R}^n$, and 
    $\Gamma \in \mConv{n}$ such that $\Gamma\le \phi$
and \eqref{defx} holds.

\noindent Find function $\Gamma^+$ such that
\begin{equation}\label{def:Gamma2}
	    \Gamma^+ \in \mConv{n}, \qquad \tau \bar \Gamma(\cdot)  + (1-\tau) [q_{\tilde f}(\cdot;x)+\tilde h(\cdot)] \le \Gamma^+(\cdot) \le \phi(\cdot),
	\end{equation}
   and where $\bar \Gamma(\cdot) $ is such that
\begin{equation}\label{def:bar Gamma2}
	\bar \Gamma \in \mConv{n}, \quad
 \bar \Gamma(x) = \Gamma(x), \quad 
 x =  \argmin_{u \in \mathbb{R}^n}\;
\bar \Gamma(u)+\frac{1}{2 \lambda}\|u-x^c\|^2.
	\end{equation}
     {\bf Output:} $\Gamma^+$.
     
    \noindent\rule[0.5ex]{1\columnwidth}{1pt}

Our quadratic bundle method QB uses blackbox
QBU and is given below.

\noindent\rule[0.5ex]{1\columnwidth}{1pt}
	
	QB
	
	\noindent\rule[0.5ex]{1\columnwidth}{1pt}
	\begin{itemize}
		\item[0.] Let $ x_0\in \dom h $, $\lambda >0$, $ \bar \varepsilon>0 $ and $\tau \in (0,1)$ satisfying
  \begin{equation}
\frac{\tau}{1-\tau} \geq \frac{8 \lambda T_{\bar \varepsilon}^2}{(1+\lambda \mu){\bar \varepsilon}}
  \end{equation}
  where $T_{\bar \varepsilon}$ is as in
  \eqref{def:T}. Set  $y_0=x_0$, $t_0=0$, and $j=1$.
\item[1.] if $t_{j-1} \leq {\bar \varepsilon}/2$
then perform a serious update, i.e., 
set 
$x_j^c=x_{j-1}$ and 
find $\Gamma_j$ such that
$$
\Gamma_{j} \in \mConv{n},\;\;q_{\tilde f}(;x_{j-1})+\tilde h \leq \Gamma_j \leq \phi,
$$
else, perform a null update, i.e., set
$x_j^c = x_{j-1}^c$ and let
$\Gamma_j$ be the output
of the QBU blackbox with inputs
$(\Gamma,x,\tau,x^c)=
(\Gamma_{j-1}, x_{j-1}, \tau, x_{j-1}^c)$;
\item[2.] compute
$$
x_j =  \argmin_{u \in \mathbb{R}^n}\;
\Gamma_j(u)+\frac{1}{2 \lambda}\|u-x_j^c\|^2,
$$
choose $y_j$ satisfying
$$
y_j \in \argmin \{\phi(x):x \in \{x_0,x_1,\ldots,x_j\}\},
$$
and set 
$$
m_j=\Gamma_j(x_j)+\frac{1}{2\lambda}\|x_j-x_j^c\|^2, \;t_j=\phi(y_j)-m_j;
$$
\item[3.] set $j \leftarrow j+1$ and go to step 1.
	\end{itemize}
	\rule[0.5ex]{1\columnwidth}{1pt}

Though QB and GPB use very different models
for the objective (the building blocks
are quadratic functions in QB while they are affine
functions in GPB), a rewriting of QB
as GPB applied to a reformulation of the problem allows us to
use the complexity analysis of GPB to obtain the
complexity analysis of QB.
This is based on the crucial observation that
the original optimization problem 
\eqref{eq:ProbIntro1} expressed in terms
of $\tilde f$ and $\tilde h$ can be reformulated
as problem \eqref{reformcomposite}
in terms of functions $f$ and $g$
given by \eqref{defftildehtilde}. Then inequality \eqref{def:Gamma2} can be 
rewritten
$$
\tau \bar \Gamma(\cdot)  + (1-\tau) [\ell_{f}(\cdot;x)+h(\cdot)] \le \Gamma^+(\cdot) \le \phi(\cdot)
$$
which shows that our QB method applied
to problem \eqref{eq:ProbIntro1} 
is the same as the GPB method from
\cite{montliang23}
applied to problem \eqref{reformcomposite}.
Therefore the complexity of QB is given by the complexity
of GPB, i.e., by Theorem 3.1 in \cite{montliang23}.
More precisely, fix 
$C>0$, $\bar \varepsilon>0$,
and assume that $f,h$
satisfy (A1)-(A3). Let $\lambda$ satisfy
$$
\frac{\bar \varepsilon}{C(M_f^2 + \bar \varepsilon L_f)}\leq \lambda \leq\frac{Cd_0^2}{\bar \varepsilon}
$$
and $\tau$ given by
$$
\tau=\left[1+\frac{(1+\lambda \mu)\bar \varepsilon}{8\lambda(M_f^2 + \bar \varepsilon L_f)}\right]^{-1}.
$$
Then QB finds an $\bar \varepsilon$-optimal solution
of \eqref{eq:ProbIntro1} in a number of iterations
bounded by 
$$
\mathcal{O}\left(\min\left\{ \frac{(M_f^2+\bar \varepsilon L_f)d_0^2}{\bar \varepsilon^2},\left(\frac{M_f^2+ \bar \varepsilon L_f}{\mu \bar \varepsilon}+1\right)\log\left(\frac{\mu d_0^2}{\bar \varepsilon}+1\right)\right\}+1\right).
$$

\section{Strong convexity of value functions}\label{sec:scvaluefunc}

Let $\|\cdot\|$ be a norm on $\mathbb{R}^m$
and let  $f: X \rightarrow \mathbb{R}$ be a function 
defined on a convex subset $X \subset \mathbb{R}^m$.
Recall that function $f$ is strongly convex on $X \subset \mathbb{R}^m$ 
with constant of strong convexity $\alpha > 0$  with respect to norm $\|\cdot\|$
iff
$$
f( t x  + (1-t)y ) \leq t f(x) + (1-t)f(y)  - \frac{\alpha t(1-t)}{2} \|y-x\|^2, 
$$
for all $0 \leq t \leq 1, x, y \in X$.

\begin{remark}
We have the following equivalent characterization of strongly convex functions.
Let $X \subset \mathbb{R}^m$ be a convex set.
Function $f: X \rightarrow \mathbb{R}$
is strongly convex on $X$ with constant of strong convexity $\alpha > 0$ 
with respect to norm $\|\cdot\|$
iff
\begin{equation}\label{characscf}
f(y) \geq f(x) + s^T (y-x)   + \frac{\alpha}{2}\|y-x\|^2, \;\forall x, y \in X, \forall s \in \partial f(x).
\end{equation}
This implies inequality \eqref{quadminorant}
used in QCC and QB.
\end{remark}

Let $X \subset \mathbb{R}^m$ and $Y \subset \mathbb{R}^n$ be two nonempty convex sets.
Let $\mathcal{A}$ be a $p \small{\times} n$ real matrix, let
$\mathcal{B}$ be a $p \small{\times} m$ real matrix, let $f: Y \small{\times} X  \rightarrow \mathbb{R}$,
and let $g:Y \small{\times} X  \rightarrow \mathbb{R}^q$.
For $b \in \mathbb{R}^p$, we define the value function
\begin{equation} \label{optclassgeneral0}
\mathcal{Q}(x)= \left\{
\begin{array}{l}
\inf f(y,x)\\
y \in \mathcal{S}(x):=\{y \in Y, \mathcal{A} y + \mathcal{B} x = b, g(y, x) \leq 0\}.
\end{array}
\right.
\end{equation}
DASC algorithm 
presented in Section \ref{sec:dasc} is based on Proposition  \ref{strongconvvfunc} below
giving conditions ensuring that $\mathcal{Q}$ is strongly convex:
\begin{prop}\label{strongconvvfunc} Consider value function $\mathcal{Q}$  given by \eqref{optclassgeneral0}.
Assume that (i) $X, Y$ are nonempty and convex sets such that $X \subseteq \mbox{dom}(\mathcal{Q})$ and
$Y$ is closed, (ii) $f, g$ are lower semicontinuous and the components $g_i$ of $g$ are convex functions.
If additionally $f$ is strongly convex on $Y \small{\times} X$ with constant of strong convexity
$\alpha$ with respect to norm $\|\cdot\|$ on $\mathbb{R}^{m+n}$, then
$\mathcal{Q}$ is strongly convex on $X$ with constant of strong convexity $\alpha$
with respect to norm $\|\cdot\|$ on $\mathbb{R}^{m}$. 
\end{prop}
\begin{proof} Take $x_1, x_2 \in X$ and $\varepsilon>0$.
Since $X \subseteq \mbox{dom}(\mathcal{Q})$ the sets $\mathcal{S}(x_1)$
and $\mathcal{S}(x_2)$ are nonempty.
Our assumptions imply that there are $y_1 \in \mathcal{S}(x_1)$ and  $y_2 \in \mathcal{S}(x_2)$ such that
$\mathcal{Q}( x_1 )\leq  f( y_1, x_1 ) \leq \mathcal{Q}( x_1 )+\varepsilon$ and 
$\mathcal{Q}( x_2 ) \leq f( y_2, x_2 ) \leq \mathcal{Q}( x_2 )+\varepsilon$. Then for every $0 \leq t \leq 1$, by convexity arguments
we have that $t y_1 + (1-t)y_2 \in \mathcal{S}(t x_1 + (1-t) x_2   )$ and therefore
$$
\begin{array}{lll}
\mathcal{Q}(  t x_1 + (1-t) x_2  ) & \leq & f( t y_1 + (1-t) y_2  , t x_1 + (1-t) x_2   )\\
& \leq & t f(y_1 , x_1) + (1-t) f( y_2, x_2 ) - \frac{1}{2} \alpha t (1-t) \| (y_2, x_2)  -   (y_1, x_1)   \|^2 \\
& \leq & t \mathcal{Q}( x_1 ) + (1-t) \mathcal{Q}(x_2) - \frac{1}{2} \alpha t (1-t) \| x_2  -  x_1   \|^2+\varepsilon.
\end{array}
$$
Passing to the limit when $\varepsilon\rightarrow 0$, we obtain
$$
\mathcal{Q}(  t x_1 + (1-t) x_2  ) \leq t \mathcal{Q}( x_1 ) + (1-t) \mathcal{Q}(x_2) - \frac{1}{2} \alpha t (1-t) \| x_2  -  x_1   \|^2,
$$
which completes the proof. \hfill
\end{proof}

\section{Multistage stochastic programs}\label{sec:pbformass}

We consider multistage stochastic optimization problems (MSPs) of the form
\begin{equation}\label{pbtosolve}
\begin{array}{l} 
\displaystyle{\inf_{x_1,\ldots,x_T}} \; \mathbb{E}_{\xi_2,\ldots,\xi_T}[ \displaystyle{\sum_{t=1}^{T}}\;f_t(x_t(\xi_1,\xi_2,\ldots,\xi_t), x_{t-1}(\xi_1,\xi_2,\ldots,\xi_{t-1}), \xi_t )]\\
x_t(\xi_1,\xi_2,\ldots,\xi_t) \in X_t( x_{t-1}(\xi_1,\xi_2,\ldots,\xi_{t-1}), \xi_t )\;\mbox{a.s.}, \;x_{t} \;\mathcal{F}_t\mbox{-measurable, }t=1,\ldots,T,
\end{array}
\end{equation}
where $x_0$ is given, $\xi_1$ is deterministic,  $(\xi_t)_{t=2}^T$ is a stochastic process, $\mathcal{F}_t$ is the sigma-algebra
$\mathcal{F}_t:=\sigma(\xi_j, j\leq t)$, and $X_t(x_{t-1}, \xi_t ),t=1,\ldots,T$, can be of two types:
\begin{itemize}
\item[(S1)] $X_t( x_{t-1}, \xi_t ) = \{x_t \in \mathbb{R}^n : x_t \in \mathcal{X}_t : x_t \geq 0,\;\;\displaystyle A_{t} x_{t} + B_{t} x_{t-1} = b_t \}$ (in this case, for short, we say that $X_t$ is of type S1); 
\item[(S2)] $X_t( x_{t-1} , \xi_t)= \{x_t \in \mathbb{R}^n : x_t \in \mathcal{X}_t,\;g_t(x_t, x_{t-1}, \xi_t) \leq 0,\;\;\displaystyle A_{t} x_{t} + B_{t} x_{t-1} = b_t \}$.
In this case, for short, we say that $X_t$ is of type S2.
\end{itemize}
For both kinds of constraints, $\xi_t$ contains in particular the random elements in matrices $A_t, B_t$, and vector $b_t$.
Note that a mix of these types of constraints is allowed: for instance we can have $X_1$ of type S1 and $X_2$ of type $S2$.\\

We make the following assumption on $(\xi_t)$:\\
\par (H0) $(\xi_t)$ 
is interstage independent and
for $t=2,\ldots,T$, $\xi_t$ is a random vector taking values in $\mathbb{R}^K$ with a discrete distribution and a
finite support $\Theta_t=\{\xi_{t 1}, \ldots, \xi_{t M}\}$ with $p_{t i}=\mathbb{P}(\xi_{t}=\xi_{t i})>0,i=1,\ldots,M$,
while $\xi_1$ is deterministic.\footnote{To alleviate notation and without loss of generality, we have assumed that the number $M$ of possible realizations
of $\xi_t$, the size $K$ of $\xi_t$, and $n$ of $x_t$ do not depend on $t$.}\\

We will denote by $A_{t j}, B_{t j},$ and $b_{t j}$ the realizations of respectively $A_t, B_t,$ and $b_t$
in $\xi_{t j}$. For this problem, we can write Dynamic Programming equations: assuming that $\xi_1$ is deterministic,
the first stage problem is 
\begin{equation}\label{firststodp}
\mathcal{Q}_1( x_0 ) = \left\{
\begin{array}{l}
\inf_{x_1 \in \mathbb{R}^n} F_1(x_1, x_0, \xi_1) := f_1(x_1, x_0, \xi_1)  + \mathcal{Q}_2 ( x_1 )\\
x_1 \in X_1( x_{0}, \xi_1 )\\
\end{array}
\right.
\end{equation}
for $x_0$ given and for $t=2,\ldots,T$, $\mathcal{Q}_t( x_{t-1} )= \mathbb{E}_{\xi_t}[ \mathfrak{Q}_t ( x_{t-1},  \xi_{t}  )  ]$ with
\begin{equation}\label{secondstodp} 
\mathfrak{Q}_t ( x_{t-1}, \xi_{t}  ) = 
\left\{ 
\begin{array}{l}
\inf_{x_t \in \mathbb{R}^n}  F_t(x_t, x_{t-1}, \xi_t ) :=  f_t ( x_t , x_{t-1}, \xi_t ) + \mathcal{Q}_{t+1} ( x_t )\\
x_t \in X_t ( x_{t-1}, \xi_t ),
\end{array}
\right.
\end{equation}
with the convention that $\mathcal{Q}_{T+1}$ is null.

We set $\mathcal{X}_0=\{x_0\}$ and make the following assumptions (H1) on the problem data: for $t=1,\ldots,T$,\\
\par (H1)-(a) for every $j=1,\ldots,M$, the function
$f_t(\cdot, \cdot,\xi_{t j})$ is strongly convex on  $\mathcal{X}_t \small{\times} \mathcal{X}_{t-1}$
with constant of strong convexity $\alpha_{t j}>0$ with respect to norm $\|\cdot\|_2$;
\par (H1)-(b) $\mathcal{X}_t$ is nonempty, convex, and compact;
\par (H1)-(c) there exists $\varepsilon_t>0$ such that for every $j=1,\ldots,M$, for every
$x_{t-1} \in \mathcal{X}_{t-1}^{\varepsilon_t}$,
the set $X_t(x_{t-1}, \xi_{t j}) \cap \mbox{ri}( \mathcal{X}_t)$ is nonempty.\\

If $X_t$ is of type $S2$ we additionally assume that:\\
\par (H1)-(d) for $t=1,\ldots,T$, there exists $\tilde \varepsilon_t>0$ such that for every $j=1,\ldots,M$, each component $g_{t i}(\cdot, \cdot, \xi_{t j}), i=1,\ldots,p$, of the function $g_t(\cdot, \cdot, \xi_{t j})$ is 
convex on $\mathcal{X}_t \small{\times} \mathcal{X}_{t-1}^{\tilde \varepsilon_t}$;
\par  (H1)-(e) for $t=2, \ldots, T$, $\forall j=1,\ldots,M$, 
there exists $({\bar x}_{t j t-1}, {\bar x}_{t j t}) \in \mathcal{X}_{t-1} \small{\times} \mbox{ri}( \mathcal{X}_t )$
such that  $A_{t j} {\bar x}_{t j t} + B_{t j} {\bar x}_{t j t-1} = b_{t j}$,
and $({\bar x}_{t j t-1}, {\bar x}_{t j t}) \in \mbox{ri}( \{ g_t(\cdot, \cdot, \xi_{t j}) \leq 0 \} )$. \\

\begin{rem} For a problem of form \eqref{pbtosolve} where the strong convexity assumption of 
functions $f_t(\cdot, \cdot,\xi_{t j})$ fails to hold, if for every $t, j$ 
function $f_t(\cdot, \cdot,\xi_{t j})$ is convex and if 
the columns
of matrix $( A_{t j} \, B_{t j})$ are independant 
 then we may reformulate the problem pushing and penalizing the linear coupling constraints
 in the objective, ending up with the strongly convex cost function
 $f_t(\cdot, \cdot,\xi_{t j}) + \rho_t \|A_{t j} x_t + B_{t j} x_{t-1} - b_{t j}  \|_2^2$
 in variables $(x_t, x_{t-1})$
 for stage $t$ realization $\xi_{t j}$ for some well chosen penalization $\rho_t >0$.
\end{rem}

\section{DASC: extension of QCC method for MSPs}\label{sec:dasc}

In this section, we introduce DASC to solve 
MSPs, which can be seen as an extension of the popular SDDP method to solve MSPs,
and which, same as QCC,
uses quadratic approximations (called cuts), to
approximate functions in the objective.

Due to Assumption (H0), the $M^{T-1}$ realizations of $(\xi_t)_{t=1}^T$ form a scenario tree of depth $T+1$
where the root node $n_0$ associated to a stage $0$ (with decision $x_0$ taken at that
node) has one child node $n_1$
associated to the first stage (with $\xi_1$ deterministic).

We denote by $\mathcal{N}$ the set of nodes, by
{\tt{Nodes}}$(t)$ the set of nodes for stage $t$ and
for a node $n$ of the tree, we define: 
\begin{itemize}
\item $C(n)$: the set of children nodes (the empty set for the leaves);
\item $x_n$: a decision taken at that node;
\item $p_n$: the transition probability from the parent node of $n$ to $n$;
\item $\xi_n$: the realization of process $(\xi_t)$ at node $n$\footnote{The same notation $\xi_{\tt{Index}}$ is used to denote
the realization of the process at node {\tt{Index}} of the scenario tree and the value of the process $(\xi_t)$
for stage {\tt{Index}}. The context will allow us to know which concept is being referred to.
In particular, letters $n$ and $m$ will only be used to refer to nodes while $t$ will be used to refer to stages.}:
for a node $n$ of stage $t$, this realization $\xi_n$ contains in particular the realizations
$b_n$ of $b_t$, $A_{n}$ of $A_{t}$, and $B_{n}$ of $B_{t}$;
\item $\xi_{[n]}$: the history of the realizations of process $(\xi_t)$ from the first stage node $n_1$ to node $n$:
 for a node $n$ of stage $t$, the $i$-th component of $\xi_{[n]}$ is $\xi_{\mathcal{P}^{t-i}(n)}$ for $i=1,\ldots, t$,
 where $\mathcal{P}:\mathcal{N} \rightarrow \mathcal{N}$ is the function 
 associating to a node its parent node (the empty set for the root node).
\end{itemize}

Similary to SDDP, at iteration $k$,  trial points $x_n^k$ are computed in a forward pass
for all nodes $n$ of the scenario tree replacing recourse functions 
$\mathcal{Q}_{t+1}$ by the approximations $\mathcal{Q}_{t+1}^{k-1}$ available at the beginning of this iteration.

In a backward pass, we then select a set of nodes $(n_1^k, n_2^k, \ldots, n_T^k)$ 
(with $n_1^k=n_1$, and for $t \geq 2$, $n_t^k$ a node of stage $t$, child of node $n_{t-1}^k$) 
corresponding to a sample $({\tilde \xi}_1^k, {\tilde \xi}_2^k,\ldots, {\tilde \xi}_T^k)$
of $(\xi_1, \xi_2,\ldots, \xi_T)$. For $t=2,\ldots,T$, a cut  
\begin{equation}\label{eqcutctk}
\mathcal{C}_t^k( x_{t-1} ) = \theta_t^{k} +  \langle \beta_t^{k}, x_{t-1}-x_{n_{t-1}^{k}}^{k} \rangle + \frac{\alpha_t}{2}\| x_{t-1}-x_{n_{t-1}^{k}}^{k} \|_2^2
\end{equation}
is computed for $\mathcal{Q}_t$ at $x_{n_{t-1}^{k}}^{k}$
where 
\begin{equation}\label{formulaalphat}
\alpha_t = \sum_{j=1}^M p_{t j} \alpha_{t j},
\end{equation}
and where the computation of coefficients $\theta_t^{k}, \beta_t^{k}$ is given below.
We show in Section \ref{convanalysis} that cut $\mathcal{C}_t^k$ is a lower bounding function
for $\mathcal{Q}_t$. Contrary to SDDP where cuts are affine functions our cuts are quadratic functions, as in QCC. 
In the end of iteration $k$, we obtain the lower approximations $\mathcal{Q}_{t}^{k}$ of $\mathcal{Q}_t,\;t=2,\ldots,T+1$, 
given by
$$
\begin{array}{lll}
\mathcal{Q}_{t}^{k}(x_{t-1}) & = &\displaystyle \max_{1 \leq \ell \leq k}\;\mathcal{C}_t^{\ell}( x_{t-1} ),
\end{array}
$$ 
which take the form of a maximum of quadratic functions.
The detailed steps of the DASC algorithm are given below.\\

\par  {\textbf{DASC, Step 1: Initialization.}} For $t=2,\ldots,T$,  take as initial approximations $\mathcal{Q}_t^0 \equiv -\infty$.  
Set $x_{n_0}^1 = x_0$, set the iteration count $k$ to 1, and $\mathcal{Q}_{T+1}^0 \equiv 0$. \\
\par {\textbf{DASC, Step 2: Forward pass.}} The forward pass performs the following computations: 
\par {\textbf{For }}$t=1,\ldots,T$,\\
\hspace*{1.6cm}{\textbf{For }}every node $n$ of stage $t-1$,\\
\hspace*{2.4cm}{\textbf{For }}every child node $m$ of node $n$, compute an optimal solution $x_m^k$ of
\begin{equation} \label{defxtkj}
{\underline{\mathfrak{Q}}}_t^{k-1}( x_n^k , \xi_m ) = \left\{
\begin{array}{l}
\displaystyle \inf_{x_m} \; F_t^{k-1}(x_m , x_n^k, \xi_m):= f_t( x_m , x_n^k , \xi_m ) + \mathcal{Q}_{t+1}^{k-1}( x_m ) \\
x_m \in X_t( x_n^k, \xi_m ),
\end{array}
\right.
\end{equation}
\hspace*{2.6cm}where $x_{n_0}^k = x_0$.\\
\hspace*{2.4cm}{\textbf{End For}}\\
\hspace*{1.6cm}{\textbf{End For}}
\par {\textbf{End For}}\\
\par  {\textbf{DASC, Step 3: Backward pass.}}
We select a set of nodes $(n_1^k, n_2^k, \ldots, n_T^k)$ 
with $n_t^k$ a node of stage $t$ ($n_1^k=n_1$ and for $t \geq 2$, $n_t^k$
a child node of $n_{t-1}^k$)
corresponding to a sample $({\tilde \xi}_1^k, {\tilde \xi}_2^k,\ldots, {\tilde \xi}_T^k)$
of $(\xi_1, \xi_2,\ldots, \xi_T)$.\\
Set $\theta_{T+1}^k=\alpha_{T+1}=0$ and $\beta_{T+1}^k=0$ which defines $\mathcal{C}_{T+1}^k \equiv 0$.\\
{\textbf{For }}$t=T,\ldots,2$,\\
\hspace*{0.8cm}{\textbf{For }}every child node $m$ of $n=n_{t-1}^k$ \\
\hspace*{1.6cm}Compute an optimal  solution $x_m^{B k}$ of
\begin{equation}\label{primalpbisddp}
{\underline{\mathfrak{Q}}}_t^k(x_n^k , \xi_m ) = \left\{
\begin{array}{l}
\displaystyle \inf_{x_m} \;F_t^k(x_m , x_n^k ,  \xi_m):=f_t( x_m , x_n^k , \xi_m) + \mathcal{Q}_{t+1}^k ( x_m )\\ 
x_m \in X_t(x_n^k , \xi_m   ).\\
\end{array}
\right.
\end{equation}
For the problem above, if $X_t$ is of type $S1$ we define the Lagrangian
$L(x_m, \lambda, \mu) = F_t^k(x_m , x_n^k ,  \xi_m) +   \lambda^T (A_m x_m + B_m x_n^k - b_m )$
and take optimal Lagrange multipliers $\lambda_m^k$.
If $X_t$ is of type $S2$ we define the Lagrangian
$L(x_m, \lambda, \mu) = F_t^k(x_m , x_n^k ,  \xi_m) +   \lambda^T (A_m x_m + B_m x_n^k - b_m ) + 
\mu^T g_t( x_m , x_n^k , \xi_m)$ and 
take optimal Lagrange multipliers $(\lambda_m^k, \mu_m^k)$.
If $X_t$ is of type $S1$,  denoting by $\mbox{SG}_{f_t( x_m^{B k}, \cdot , \xi_m )}( x_n^{k} )$
a subgradient of convex function $f_t( x_m^{B k}, \cdot , \xi_m  )$ at $x_n^{k}$,
we compute $\theta_{t}^{k m} = {\underline{\mathfrak{Q}}}_t^k(x_n^k , \xi_m )$
and
$$
\beta^{k m}=\mbox{SG}_{f_t( x_m^{B k } ,\cdot , \xi_m  )}( x_n^{k} )+ B_m^T \lambda_m^k.
$$
If $X_t$ is of type $S2$ denoting by 
$\mbox{SG}_{g_{t i} (x_m^{B k}, \cdot , \xi_m )}(x_n^{k})$ a subgradient of convex function $g_{t i} ( x_m^{B k} , \cdot , \xi_m )$
at $x_n^{k}$ we compute $\theta_{t}^{k m} = {\underline{\mathfrak{Q}}}_t^k(x_n^k , \xi_m )$ and
$$
\beta^{k m}=\mbox{SG}_{f_t( x_m^{B k } ,\cdot , \xi_m  )}( x_n^{k} )+ B_m^T \lambda_m^k + \sum_{i=1}^p \mu_m^k(i) \mbox{SG}_{g_{t i} ( x_m^{B k } ,\cdot , \xi_m  )}(x_n^{k}).
$$
\hspace*{0.8cm}{\textbf{End For}}\\
\hspace*{0.8cm}The new cut $\mathcal{C}_t^k$ is obtained  computing
\begin{equation}\label{formulathetak}
\theta_t^k=\sum_{m \in C(n)} p_m \theta_{t}^{k m} \mbox{  and  } \beta_t^k=\sum_{m \in C(n)} p_m  \beta^{k m}.
\end{equation}
{\textbf{End For}}\\
\par {\textbf{DASC, Step 4:}} Do $k \leftarrow k+1$ and go to Step 2.

\begin{rem} In DASC, decisions are computed at every iteration for all the nodes of the scenario tree
in the forward pass.
However, in practice, sampling will be used in the forward pass to compute
at iteration $k$  decisions only for the nodes $(n_1^k,\ldots,n_T^k)$
and their children nodes. 
The variant of DASC written above is convenient for the convergence analysis of the method, presented in the 
next section. From this convergence analysis, it is possible to show the convergence of the variant
of DASC which uses sampling in the forward pass (see also Remark 5.4 in \cite{guigues2016isddp}
and Remark 4.3 in  \cite{guilejtekregsddp}).
\end{rem} 

\begin{rem} By change of variable we can write $\mathcal{Q}_t^k$ under the form
$$
\mathcal{Q}_t^k ( x_{t-1} ) = \frac{\alpha_t}{2}\|x_{t-1} - {\bar x}_{t-1}^k\|_2^2 + \max_{1 \leq j \leq k} \;{\tilde \theta}_t^k + \langle \tilde \beta_t^k , x_{t-1} \rangle.
$$
Therefore all subproblems can be reformulated as quadratic programs with a quadratic term multiple of $\|x-\bar x\|_2^2$
and linear constraints (the same number as with SDDP).
\end{rem}

\section{Convergence analysis of DASC}\label{convanalysis}

In Theorem \ref{convsddpsconv} below we show the convergence of DASC making the following additional assumption:\\

\par (H2) The samples in the backward passes are independent: $(\tilde \xi_2^k, \ldots, \tilde \xi_T^k)$ is a realization of
$\xi^k=(\xi_2^k, \ldots, \xi_T^k) \sim (\xi_2, \ldots,\xi_T)$ 
and $\xi^1, \xi^2,\ldots,$ are independent.\\

\par We will make use of the following lemma:
\begin{lemma} \label{convrecfuncQtS} Let Assumptions (H0) and (H1) hold. Then for $t=2,\ldots,T+1$, function $\mathcal{Q}_t$ is convex and Lipschitz continuous on 
$\mathcal{X}_{t-1}$.
\end{lemma}
\begin{proof} The proof is analogue to the proofs of Lemma 3.2 in \cite{guiguessiopt2016} and
Lemma 2.2 in \cite{lecphilgirar12}.\hfill
\end{proof}

\begin{thm} \label{convsddpsconv}
Consider the sequences of stochastic decisions $x_n^k$ and of recourse functions $\mathcal{Q}_ t^k$
generated by DASC.
Let Assumptions (H0), (H1) and (H2) hold. Then
\begin{itemize}
\item[(i)] almost surely, for $t=2,\ldots,T+1$, the following holds:
$$
\mathcal{H}(t): \;\;\;\forall n \in {\tt{Nodes}}(t-1), \;\; \displaystyle \lim_{k \rightarrow +\infty} \mathcal{Q}_{t}(x_{n}^{k})-\mathcal{Q}_{t}^{k}(x_{n}^{k} )=0.
$$
\item[(ii)]
Almost surely, the limit of the sequence
$( {F}_1^{k-1}(x_{n_1}^k , x_0 ,  \xi_1) )_k$ of the approximate first stage optimal values
and of the sequence
$({\underline{\mathfrak{Q}}}_1^{k}(x_{0}, \xi_1))_k$
is the optimal value 
$\mathcal{Q}_{1}(x_0)$ of \eqref{pbtosolve}.
Let $\Omega=(\Theta_2 \small{\times} \ldots \small{\times} \Theta_T)^{\infty}$ be the sample space
of all possible sequences of scenarios equipped with the product $\mathbb{P}$ of the corresponding 
probability measures. Define on $\Omega$ the random variable 
$x^* = (x_1^*, \ldots, x_T^*)$ as follows. For $\omega \in \Omega$, consider the 
corresponding sequence of decisions $( (x_n^k( \omega ))_{n \in \mathcal{N}} )_{k \geq 1}$
computed by DASC. Take any accumulation point
$(x_n^* (\omega) )_{n \in \mathcal{N}}$ of this sequence. 
If $\mathcal{Z}_t$ is the set of $\mathcal{F}_t$-measurable functions,
define $x_1^*(\omega),\ldots,x_T^*(\omega)$ taking $x_t^{*}(\omega): \mathcal{Z}_t \rightarrow \mathbb{R}^n$ given by
$x_t^{*}(\omega)( \xi_1, \ldots, \xi_t  )=x_{m}^{*}(\omega)$ where $m$ is given by $\xi_{[m]}=(\xi_1,\ldots,\xi_t)$ for $t=1,\ldots,T$.
Then $\mathbb{P}((x_1^*,\ldots,x_T^*) \mbox{ is an optimal solution to \eqref{pbtosolve}})  =1$.
\end{itemize}
\end{thm}
\begin{proof} 
\par Let us prove (i). We first check by induction on $k$ and
backward induction on $t$ that for all $k \geq 0$,
for all $t=2,\ldots,T+1$, for any node $n$ of stage $t-1$ and decision $x_n$ taken at that node we have
\begin{equation}\label{validcut}
\mathcal{Q}_t ( x_{n} ) \geq \mathcal{C}_t^k ( x_{n} ),
\end{equation}
almost surely.
For any fixed $k$, relation \eqref{validcut} holds for $t=T+1$ and if it holds until iteration $k$ for 
$t+1$ 
with  $t \in \{2,\ldots,T\}$, we deduce that for any node $n$ of stage $t-1$ and decision $x_n$ taken at that node we have
$\mathcal{Q}_{t+1}( x_n ) \geq \mathcal{Q}_{t+1}^k ( x_n )$, 
$\mathfrak{Q}_t( x_n , \xi_m ) \geq  {\underline{\mathfrak{Q}}}_t^{k}( x_n , \xi_m )$
for any child node $m$ of $n$.
Now note that function $(x_m,x_n) \rightarrow \mathcal{Q}_{t+1}^k(x_m)$ is convex (as a maximum of convex functions) and recalling
that $(x_m,x_n) \rightarrow f_t(x_m,x_n,\xi_m)$ is strongly convex with constant of strong convexity
$\alpha_{t m}$, the function $(x_m,x_n) \rightarrow f_t(x_m,x_n,\xi_m) +\mathcal{Q}_{t+1}^k( x_m )$ is also
strongly convex with the same parameter of strong convexity. Using Proposition \ref{strongconvvfunc},
it follows that ${\underline{\mathfrak{Q}}}_t^{k}( \cdot , \xi_m )$  is strongly convex  
with constant of strong convexity $\alpha_{t m}$.
Using Lemma 2.1 in \cite{guiguessiopt2016} we have that
$\beta^{k m} \in \partial  {\underline{\mathfrak{Q}}}_t^k(\cdot , \xi_m )(x_{n_{t-1}^k}^{k})$.
Recalling characterization \eqref{characscf} of strongly convex functions,
we get for any $x_n \in \mathcal{X}_{t-1}$:
$$
\begin{array}{lll}
{\underline{\mathfrak{Q}}}_t^{k}( x_n , \xi_m ) & \geq &  {\underline{\mathfrak{Q}}}_t^{k}( x_{n_{t-1}^k}^k , \xi_m )+ \langle \beta^{k m} , x_n - x_{n_{t-1}^k}^k \rangle + \frac{\alpha_{t m}}{2} \|x_n- x_{n_{t-1}^k}^k\|_2^2
\end{array} 
$$
and therefore for any node $n$ of stage $t-1$  and decision $x_n$ taken at that node we have 
\begin{equation}\label{checkvalidcut}
\begin{array}{lll}
\mathcal{Q}_t( x_n ) & = & \displaystyle \sum_{m \in C(n)} p_m \mathfrak{Q}_t( x_n , \xi_m )\\
& \geq  & \displaystyle \sum_{m \in C(n)} p_m {\underline{\mathfrak{Q}}}_t^{k}( x_n , \xi_m )\\
& \geq & \displaystyle \sum_{m \in C(n)} p_m \Big(    {\underline{\mathfrak{Q}}}_t^{k}( x_{n_{t-1}^k}^k , \xi_m )
+ \langle \beta^{k m} , x_n - x_{n_{t-1}^k}^k \rangle + \frac{\alpha_{t m}}{2} \|x_n -x_{n_{t-1}^k}^k\|_2^2  \Big)\\
& = & \theta_t^{k} +  \langle \beta_t^{k}, x_{n}-x_{n_{t-1}^{k}}^k \rangle + \frac{\alpha_t}{2}\| x_{n}-x_{n_{t-1}^{k}}^k \|_2^2 \\
& = & \mathcal{C}_t^k (x_n). 
\end{array} 
\end{equation}
This completes the induction step and shows \eqref{validcut} for every $t, k$.

\par Let $\Omega_1$ be the event on the sample space $\Omega$  of sequences
of scenarios such that every scenario is sampled an infinite number of times.
Due to (H2), this event has probability one.
Take an arbitrary  realization $\omega$ of DASC in $\Omega_1$. 
To simplify notation we will use $x_n^k, \mathcal{Q}_t^k, \theta_t^k, \beta_t^k$ instead 
of $x_n^k(\omega), \mathcal{Q}_t^k(\omega), \theta_t^k(\omega), \beta_t^k( \omega )$.\\
We want to show that $\mathcal{H}(t), t=2,\ldots,T+1$, hold for that realization.
The proof is by backward induction on $t$. For $t=T+1$, $\mathcal{H}(t)$ holds
by definition of $\mathcal{Q}_{T+1}$, $\mathcal{Q}_{T+1}^k$. Now assume that $\mathcal{H}(t+1)$ holds
for some $t \in \{2,\ldots,T\}$. We want to show that $\mathcal{H}(t)$ holds.
Take an arbitrary node $n \in {\tt{Nodes}}(t-1)$. For this node we define 
$\mathcal{S}_n=\{k \geq 1 : n_{t-1}^k = n\}$ the set of iterations such that the sampled scenario passes through node $n$.
Observe that $\mathcal{S}_n$ is infinite because the realization of DASC is in $\Omega_1$.
We first show that 
$$
\displaystyle \lim_{k \rightarrow +\infty, k \in \mathcal{S}_n } \mathcal{Q}_{t}(x_{n}^{k})-\mathcal{Q}_{t}^{k}(x_{n}^{k} )=0. 
$$
For $k \in \mathcal{S}_n$, we have $n_{t-1}^k =n$, i.e., $x_n^k = x_{n_{t-1}^k}^k$, which implies, using \eqref{validcut}, that
\begin{equation}\label{firsteqstosddp}
\mathcal{Q}_t ( x_n^k ) \geq \mathcal{Q}_t^k ( x_n^k ) \geq  \mathcal{C}_t^k ( x_n^k ) = \theta_t^k =  \sum_{m \in C(n)} p_m \theta_{t}^{k m} 
= \sum_{m \in C(n)} p_m  {\underline{\mathfrak{Q}}}_t^k(x_n^k , \xi_m )
\end{equation}
by definition of $\mathcal{C}_t^k$ and $\theta_t^k$.
It follows that for any $k \in \mathcal{S}_n$ we have
\begin{equation} \label{eqconv1bisfutures}
\begin{array}{lll}
0 \leq \mathcal{Q}_{t}(x_{n}^k) - \mathcal{Q}_{t}^k(x_{n}^k) & \leq & 
\displaystyle \sum_{m \in C(n)} p_m \Big( \mathfrak{Q}_t(x_n^k , \xi_m )  -    {\underline{\mathfrak{Q}}}_t^k(x_n^k , \xi_m )    \Big) \\ 
&  \leq & 
\displaystyle \sum_{m \in C(n)} p_m \Big( \mathfrak{Q}_t(x_n^k , \xi_m )  -    {\underline{\mathfrak{Q}}}_t^{k-1}(x_n^k , \xi_m )    \Big)  \\
&  = & 
\displaystyle \sum_{m \in C(n)} p_m \Big( \mathfrak{Q}_t(x_n^k , \xi_m )  -   F_t^{k-1}(x_m^{k}, x_{n}^k, \xi_m )     \Big)  \\ 
&  = & 
\displaystyle \sum_{m \in C(n)} p_m \Big( \mathfrak{Q}_t(x_n^k , \xi_m )  -   f_t(x_m^{k}, x_{n}^k, \xi_m ) -  \mathcal{Q}_{t+1}^{k-1}( x_m^k )     \Big)  \\ 
&  = & 
\displaystyle \sum_{m \in C(n)} p_m \Big( \mathfrak{Q}_t(x_n^k , \xi_m )  -   F_t(x_m^{k}, x_{n}^k, \xi_m ) + \mathcal{Q}_{t+1}( x_m^k )  -  \mathcal{Q}_{t+1}^{k-1}( x_m^k )     \Big)  \\ 
 & \leq  & \displaystyle \sum_{m \in C(n)} p_m \Big(  \mathcal{Q}_{t+1}( x_m^k ) - \mathcal{Q}_{t+1}^{k-1}( x_m^k ) \Big),
\end{array}
\end{equation}
where for the last inequality we have used the definition of $\mathfrak{Q}_t$  and the fact that $x_m^k \in X_t ( x_n^k , \xi_m )$.

Next, recall that $\mathcal{Q}_{t+1}$ is convex; by Lemma \ref{convrecfuncQtS} functions $(\mathcal{Q}_{t+1}^k)_k$ are Lipschitz continuous;
and for all $k \geq 1$ we have $\mathcal{Q}_{t+1}^k \leq \mathcal{Q}_{t+1}^{k+1} \leq\mathcal{Q}_{t+1}$ on compact set $\mathcal{X}_t$.
Therefore, the induction hypothesis
$$
\lim_{k \rightarrow +\infty} \mathcal{Q}_{t+1}( x_m^k ) - \mathcal{Q}_{t+1}^k ( x_m^k )=0 
$$
implies, using Lemma A.1 in \cite{lecphilgirar12}, that
\begin{equation}\label{indchypreformulatedsto}
\lim_{k \rightarrow +\infty} \mathcal{Q}_{t+1}( x_m^k ) - \mathcal{Q}_{t+1}^{k-1} ( x_m^k )=0 . 
\end{equation}

Plugging \eqref{indchypreformulatedsto} into
\eqref{eqconv1bisfutures} we obtain 
\begin{equation}\label{mainisddp}
\displaystyle \lim_{k \rightarrow +\infty, k \in \mathcal{S}_n} \mathcal{Q}_{t}(x_{n}^{k})-\mathcal{Q}_{t}^{k}(x_{n}^{k} )=0. 
\end{equation}
It remains to show that
\begin{equation}\label{alsoconvnotinsnisddp}
\displaystyle \lim_{k \rightarrow +\infty, k \notin \mathcal{S}_n } \mathcal{Q}_{t}(x_{n}^{k})-\mathcal{Q}_{t}^{k}(x_{n}^{k} )=0. 
\end{equation}
The relation above can be proved using Lemma 5.4 in \cite{guilejtekregsddp} which can be applied since 
(A) relation \eqref{mainisddp} holds (convergence was shown for the iterations in $\mathcal{S}_n$),
(B) the sequence $(\mathcal{Q}_t^k)_k$ is monotone, i.e., 
$\mathcal{Q}_t^k \geq \mathcal{Q}_t^{k-1}$ for all $k \geq 1$, (C) Assumption (H2) holds, and
(D) $\xi_{t-1}^k$ is independent on $( (x_{n}^j,j=1,\ldots,k), (\mathcal{Q}_{t}^j,j=1,\ldots,k-1))$.\footnote{Lemma 5.4 in \cite{guilejtekregsddp} is similar to the end of the proof of Theorem 4.1 in \cite{guiguessiopt2016} and uses
the Strong Law of Large Numbers. This lemma itself applies the ideas of the end of the convergence proof of SDDP given in \cite{lecphilgirar12}, which
was given with a different (more general) sampling scheme in the backward pass.} Therefore, we have shown (i).\\

\par (ii) can be proved as Theorem 5.3-(ii) in \cite{guigues2016isddp} using (i).\hfill
\end{proof}

\section{Numerical experiments}\label{sec:num}

Consider the Dynamic Programming equations
given by
$\mathcal{Q}_t( x_{t-1} )=\mathbb{E}_{\xi_t}[\mathfrak{Q}_t(x_{t-1},\xi_t)]$
for $t=1,\ldots,T$, $x_0$ given,
$\mathcal{Q}_{T+1} \equiv 0$, where for
$t=1,\ldots,T$, we have
$$
\mathfrak{Q}_t(x_{t-1},\xi_t) =
\left\{ 
\begin{array}{l}
\displaystyle \inf_{x_t \in \mathbb{R}^n} 
\frac{1}{2}\left( 
\begin{array}{c}
x_{t-1}\\
x_t
\end{array}
\right)^T \left(\xi_t \xi_t^T + \lambda_0 I \right)
\left( 
\begin{array}{c}
x_{t-1}\\
x_t
\end{array}
\right) + 
\xi_t^T \left( 
\begin{array}{c}
x_{t-1}\\
x_t
\end{array}
\right) + \mathcal{Q}_{t+1} ( x_t )\\
x_t \geq 0, \displaystyle \sum_{i=1}^n x_t(i)=1.
\end{array}
\right.
$$

Recourse functions 
$\mathcal{Q}_t$ are strongly convex
with parameter $\lambda_0$ for $\|\cdot\|_2$ and therefore
DASC and SDDP can be applied to solve this problem.
We run DASC and SDDP
for several values of 
$T$, $n$ (the common size
of $x_{t-1}$ and $x_t$), $M$ (the number of possible realizations
of $\xi_t$ for every $t$),
and $\lambda_0$. Entries in realizations $\xi_{ti}$ of $\xi_t$ are drawn independently
from the uniform distribution in $[0,1]$.
As in SDDP, we can compute
for DASC at every iteration a lower
bound LB on the optimal value of the
problem which is the optimal value
of the approximate first stage problem
obtained replacing $\mathcal{Q}_2$
by its approximation $\mathcal{Q}_2^k$.
Similarly to SDDP, we also compute
with DASC a statistical upper bound UB. 
All iterations are run with
a single scenario in the forward pass.
We use the stopping test
from the numerical experiments
in \cite{vinalcheng}
computing the statistical upper bound from
the total cost (from $t=1$
to $t=T$) on the last 200 scenarios:
the algorithm stops when
(UB-LB)/UB$\leq \varepsilon$
with $\varepsilon=0.1$.

We report in Table \ref{tableres}
the CPU time (in seconds) needed to solve the problem and the lower (LB in the table) and upper (UB in the table) bounds
at termination. 

\begin{table}
\centering
\begin{tabular}{|c|c|c|c|c|c|c|c|}
 \hline
 T &  n & M & $\lambda_0$ & \begin{tabular}{c}DASC\\ (CPU)\end{tabular} & \begin{tabular}{c}SDDP\\ (CPU)\end{tabular} & DASC LB/UB & SDDP LB/UB\\
 \hline
10 &50 &10&$10^3$&1493&1187 &25 474/26 161&25 554/26 521\\
 \hline
10 &50&10&1&779&683 &382/417&382/418\\
 \hline
 4&100&5&$10^5$&84&418 &5.02x$10^5$/5.04x$10^5$&5.02x$10^5$/5.04x$10^5$\\
 \hline
  4&100&5&$10^6$&71&162 &5.004x$10^7$/5.008x$10^7$&5.004x$10^7$/5.009x$10^7$\\
 \hline
  3&500&5&$10^6$&543&1560 &2.5004x$10^8$/2.5004x$10^8$&2.5004x$10^8$/2.5004x$10^8$\\
 \hline
 3&600&5&$10^6$&984&4116 &3.0005x$10^8$/3.0005x$10^8$&3.0005x$10^8$/3.0005x$10^8$\\
 \hline
  3&200&5&$10^5$&141&791 &1.0006x$10^7$/1.0008.x$10^7$&1.0006.x$10^7$/1.0008x$10^7$\\
 \hline
\end{tabular}
\caption{CPU time (in seconds) and upper and lower bounds at termination for several instances with DASC and SDDP}\label{tableres}
\end{table}

The Matlab code of this implementation is availble
at \url{https://github.com/vguigues/DASC}
on github. In this implementation, linear and quadratic subproblems in DASC and SDDP were solved using Gurobi.
The methods were run on an Intel Core i7, 1.8GHz, processor with 12,0 Go of
RAM. We observe convergence of the upper and lower
bounds to the optimal value and at termination close approximate
optimal values for SDDP and DASC (as a check for
correctness of the implementations).
The CPU times are much smaller with DASC when the parameter $\lambda_0$
is large; otherwise they are of the same 
order of magnitude.
We indeed expect
that when the constant of strong convexity of $\mathcal{Q}_t$ is large ($\lambda_0$ for our instance) DASC provides much better approximations of these functions
and will converge quicker (providing "better" approximations and therefore "good" trial points quicker).

\section{Conclusion and extensions}\label{sec:conc}

We  introduced two new methods for convex
deterministic optimization problems: QCC (Quadratic Cuts
for Convex optimization) and QB (Quadratic Bundle method). We proved the complexity of
these methods for composite optimization problems which are the sum of a convex function $\tilde h$
and of a strongly convex function $\tilde f$. 
We extended this idea
of using quadratic approximations in the objective to build models for the recourse functions
of MSPs when the strong convexity assumption holds
for these functions, yielding DASC algorithm. We also proved
the convergence of DASC.

A number of interesting questions arise
from our developments:
\begin{itemize}
\item[(A)]  Variants of 
QCC where $\tau$ varies along iterations could be studied as well as strategies to choose parameter $\tau=\tau_j$
at iteration $j$.
\item[(B)] Similarly, variants of 
QB where $\tau$ and $\lambda$ vary  along iterations could be studied as well as strategies to choose parameters $\tau=\tau_j$ and
$\lambda=\lambda_j$
at iteration $j$.
\item[(C)] DASC applies to the case where recourse functions in MSPs are strongly convex. For linear multistage stochastic programs and discrete uncertainty, this assumption does not hold.
However, for linear two-stage stochastic programs and
continuous distributions, conditions are given in \cite{schultz1994}
ensuring strong convexity
of the second stage recourse function.
It is natural to study the extension of this result to MSPs: what conditions ensure
strong convexity (and for which  strong convexity parameter) of
recourse functions of linear MSPs having more than 2 stages?
For such problems, another task will therefore be to extend DASC
to the case when distributions of $\xi_t$ are continuous. 
\end{itemize}

\addcontentsline{toc}{section}{References}
\bibliographystyle{plain}
\bibliography{DASC}

\begin{thebibliography}{10}

\bibitem{alamo2019gradient}
T.~Alamo, P.~Krupa, and D.~Limon.
\newblock Gradient based restart {FISTA}.
\newblock In {\em 2019 IEEE 58th Conference on Decision and Control (CDC)}, pages 3936--3941. IEEE, 2019.

\bibitem{alamo2022restart}
T.~Alamo, P.~Krupa, and D.~Limon.
\newblock Restart of accelerated first-order methods with linear convergence under a quadratic functional growth condition.
\newblock {\em IEEE Transactions on Automatic Control}, 68(1):612--619, 2022.

\bibitem{alamo2019restart}
T.~Alamo, D.~Limon, and P.~Krupa.
\newblock Restart {FISTA} with global linear convergence.
\newblock In {\em 2019 18th European Control Conference (ECC)}, pages 1969--1974. IEEE, 2019.

\bibitem{aujol2023parameter}
J-F Aujol, L.~Calatroni, C.~Dossal, H.~Labarri{\`e}re, and A.~Rondepierre.
\newblock Parameter-free {FISTA} by adaptive restart and backtracking.
\newblock {\em arXiv preprint arXiv:2307.14323}, 2023.

\bibitem{aujol2023fista}
J-F Aujol, C.~Dossal, and A.~Rondepierre.
\newblock {FISTA} is an automatic geometrically optimized algorithm for strongly convex functions.
\newblock {\em Mathematical Programming}, pages 1--43, 2023.

\bibitem{aujol2022fista}
J-F Aujol, C.~H. Dossal, H.~Labarri{\`e}re, and A.~Rondepierre.
\newblock {FISTA} restart using an automatic estimation of the growth parameter.
\newblock {\em Preprint}, 2022.

\bibitem{birgemulti}
J.R. Birge.
\newblock Decomposition and partitioning methods for multistage stochastic linear programs.
\newblock {\em Oper. Res.}, 33:989--1007, 1985.

\bibitem{birgedono}
J.R. Birge and C.~J. Donohue.
\newblock {The Abridged Nested Decomposition Method for Multistage Stochastic Linear Programs with Relatively Complete Recourse}.
\newblock {\em Algorithmic of Operations Research}, 1:20--30, 2001.

\bibitem{chenpowell99}
Z.L. Chen and W.B. Powell.
\newblock {Convergent Cutting-Plane and Partial-Sampling Algorithm for Multistage Stochastic Linear Programs with Recourse}.
\newblock {\em {J. Optim. Theory Appl.}}, 102:497--524, 1999.

\bibitem{fercoq2019adaptive}
O.~Fercoq and Z.~Qu.
\newblock Adaptive restart of accelerated gradient methods under local quadratic growth condition.
\newblock {\em IMA Journal of Numerical Analysis}, 39(4):2069--2095, 2019.

\bibitem{lecphilgirar12}
P.~Girardeau, V.~Leclere, and A.B. Philpott.
\newblock On the convergence of decomposition methods for multistage stochastic convex programs.
\newblock {\em Mathematics of Operations Research}, 40:130--145, 2015.

\bibitem{guiguescoap2013}
V.~Guigues.
\newblock {SDDP for some interstage dependent risk-averse problems and application to hydro-thermal planning}.
\newblock {\em Computational Optimization and Applications}, 57:167--203, 2014.

\bibitem{guiguessiopt2016}
V.~Guigues.
\newblock Convergence analysis of sampling-based decomposition methods for risk-averse multistage stochastic convex programs.
\newblock {\em SIAM Journal on Optimization}, 26:2468--2494, 2016.

\bibitem{guigues2016isddp}
V.~Guigues.
\newblock {Inexact decomposition methods for solving deterministic and stochastic convex dynamic programming equations}.
\newblock {\em arXiv}, 2017.
\newblock \url{https://arxiv.org/abs/1707.00812}.

\bibitem{guisvmont}
V.~Guigues.
\newblock {Inexact cuts in SDDP applied to multistage stochastic nondifferentiable problems}.
\newblock {\em {Siam Journal on Optimization}}, 31:2084–2110, 2021.

\bibitem{gui21r}
V.~Guigues.
\newblock {Multistage stochastic programs with a random number of stages: dynamic programming equations, solution methods, and application to portfolio selection}.
\newblock {\em {Optimization Methods \& Software}}, 36:211--236, 2021.

\bibitem{universalguilimont}
V.~Guigues, J.~Liang, and R.~Monteiro.
\newblock Universal subgradient and proximal bundle methods for convex and strongly convex hybrid composite optimization.
\newblock {\em arXiv}, 2024.

\bibitem{guiguesrom10}
V.~Guigues and W.~R\"omisch.
\newblock Sampling-based decomposition methods for multistage stochastic programs based on extended polyhedral risk measures.
\newblock {\em SIAM J. Optim.}, 22:286--312, 2012.

\bibitem{guiguesrom12}
V.~Guigues and W.~R\"omisch.
\newblock {SDDP for multistage stochastic linear programs based on spectral risk measures}.
\newblock {\em Operations Research Letters}, 40:313--318, 2012.

\bibitem{guishapcheng}
V.~Guigues, A.~Shapiro, and Y.~Cheng.
\newblock {Risk-Averse Stochastic Optimal Control: an efficiently computable statistical upper bound}.
\newblock {\em {Operations Research Letters}}, 51:393--400, 2023.

\bibitem{guilejtekregsddp}
V.~Guigues, W.~Tekaya, and M.~Lejeune.
\newblock Regularized stochastic dual dynamic programming for convex nonlinear optimization problems.
\newblock {\em Optimization and Engineering}, 21:1133--1165, 2020.

\bibitem{resa}
M.~Hindsberger and A.~B. Philpott.
\newblock Resa: A method for solving multi-stage stochastic linear programs.
\newblock {\em SPIX Stochastic Programming Symposium}, 2001.

\bibitem{kelley66}
Kelley J.E.
\newblock The cutting-plane method for solving convex programs.
\newblock {\em Journal of the Society for Industrial and Applied Mathematics}, 8(4):703--712, 1960.

\bibitem{kiwiel1986method}
K.~C. Kiwiel.
\newblock A method for solving certain quadratic programming problems arising in nonsmooth optimization.
\newblock {\em IMA Journal of Numerical Analysis}, 6(2):137--152, 1986.

\bibitem{kiwiel1995proximal}
K.~C. Kiwiel.
\newblock Proximal level bundle methods for convex nondifferentiable optimization, saddle-point problems and variational inequalities.
\newblock {\em Mathematical Programming}, 69(1-3):89--109, 1995.

\bibitem{kiwiel2000efficiency}
K.~C. Kiwiel.
\newblock Efficiency of proximal bundle methods.
\newblock {\em Journal of Optimization Theory and Applications}, 104(3):589--603, 2000.

\bibitem{kiwiel2006methods}
K.~C. Kiwiel.
\newblock {\em Methods of descent for nondifferentiable optimization}, volume 1133.
\newblock Springer, 2006.

\bibitem{kozmikmorton}
V.~Kozmik and D.P. Morton.
\newblock {Evaluating policies in risk-averse multi-stage stochastic programming}.
\newblock {\em {Mathematical Programming}}, 152:275--300, 2015.

\bibitem{lan2015bundle}
G.~Lan.
\newblock Bundle-level type methods uniformly optimal for smooth and nonsmooth convex optimization.
\newblock {\em Mathematical Programming}, 149(1-2):1--45, 2015.

\bibitem{lan2023optimal}
G.~Lan, Y.~Ouyang, and Z.~Zhang.
\newblock Optimal and parameter-free gradient minimization methods for convex and nonconvex optimization.
\newblock {\em arXiv: 2310.12139}, 2023.

\bibitem{lemarechal1995new}
C.~Lemar{\'e}chal, A.~Nemirovski, and Y.~Nesterov.
\newblock New variants of bundle methods.
\newblock {\em Mathematical programming}, 69(1-3):111--147, 1995.

\bibitem{bundlegpbstochastic}
J.~Liang, V.~Guigues, and R.~Monteiro.
\newblock A single cut proximal bundle method for stochastic convex composite optimization.
\newblock {\em Mathematical programming}, 2023.

\bibitem{montliang23}
J.~Liang and R.~Monteiro.
\newblock A unified analysis of a class of proximal bundle methods for solving hybrid convex composite optimization problems.
\newblock {\em Mathematics of Operations Research}, 49:832–855, 2023.

\bibitem{mishchenko2020adaptive}
K.~Mishchenko and Y.~Malitsky.
\newblock Adaptive gradient descent without descent.
\newblock In {\em 37th International Conference on Machine Learning (ICLM 2020)}, 2020.

\bibitem{nesterov2013gradient}
Y.~Nesterov.
\newblock Gradient methods for minimizing composite functions.
\newblock {\em Mathematical Programming}, 140(1):125--161, 2013.

\bibitem{nesterov2015universal}
Y.~Nesterov.
\newblock Universal gradient methods for convex optimization problems.
\newblock {\em Mathematical Programming}, 152(1):381--404, 2015.

\bibitem{pereira}
M.V.F. Pereira and L.M.V.G Pinto.
\newblock Multi-stage stochastic optimization applied to energy planning.
\newblock {\em Math. Program.}, 52:359--375, 1991.

\bibitem{philpmatos}
A.~Philpott and V.~de~Matos.
\newblock Dynamic sampling algorithms for multi-stage stochastic programs with risk aversion.
\newblock {\em European Journal of Operational Research}, 218:470--483, 2012.

\bibitem{philpot}
A.~B. Philpott and Z.~Guan.
\newblock On the convergence of stochastic dual dynamic programming and related methods.
\newblock {\em Oper. Res. Lett.}, 36:450--455, 2008.

\bibitem{renegar2022simple}
J.~Renegar and B.~Grimmer.
\newblock A simple nearly optimal restart scheme for speeding up first-order methods.
\newblock {\em Foundations of Computational Mathematics}, 22(1):211--256, 2022.

\bibitem{schultz1994}
R.~Schultz.
\newblock Strong convexity in stochastic programs with complete recourse.
\newblock {\em Journal of Computational and Applied Mathematics}, 56:3--22, 1994.

\bibitem{shapsddp}
A.~Shapiro.
\newblock Analysis of stochastic dual dynamic programming method.
\newblock {\em European Journal of Operational Research}, 209:63--72, 2011.

\bibitem{shadenrbook}
A.~Shapiro, D.~Dentcheva, and A.~Ruszczy\'nski.
\newblock {\em {Lectures on Stochastic Programming: Modeling and Theory}}.
\newblock SIAM, Philadelphia, 2009.

\bibitem{vinalcheng}
Y.~Cheng V.~Guigues, A.~Shapiro.
\newblock Duality and sensitivity analysis of multistage linear stochastic programs.
\newblock {\em European journal of Operational Research}, 308(2):752--767, 2023.

\bibitem{zhou2024adabb}
D.~Zhou, S.~Ma, and J.~Yang.
\newblock Ada{BB}: Adaptive {B}arzilai-{B}orwein method for convex optimization.
\newblock {\em arXiv:2401.08024}, 2024.

\end{thebibliography}

\end{document}